\documentclass{amsart}

\usepackage{amssymb}

\usepackage[all]{xy}

\newtheorem{thm}{Theorem}

\newtheorem{lem}[thm]{Lemma}
\newtheorem{cor}[thm]{Corollary}

\newtheorem{prop}[thm]{Proposition}

\theoremstyle{definition}
\newtheorem{defn}[thm]{Definition}

\newtheorem{say}[thm]{}
\newtheorem{exmp}[thm]{Example}

\newtheorem{ques}[thm]{Question}    

\newtheorem*{ack}{Acknowledgments}      

\newtheorem{defn-thm}[thm]{Definition--Theorem}  
\newtheorem{defn-lem}[thm]{Definition--Lemma}  

\theoremstyle{remark}

\renewcommand{\c}[0]{{\mathbb C}}  

\renewcommand{\o}[0]{{\mathcal O}} 
\newcommand{\z}[0]{{\mathbb Z}}

\renewcommand{\r}[0]{{\mathbb R}}

\newcommand{\p}[0]{{\mathbb P}}

\newcommand{\q}[0]{{\mathbb Q}}
\newcommand{\map}[0]{\dasharrow}

\newcommand{\qtq}[1]{\quad\mbox{#1}\quad}

\newcommand{\mult}[0]{\operatorname{mult}}

\newcommand{\supp}[0]{\operatorname{Supp}}    
    
\newcommand{\codim}[0]{\operatorname{codim}}    
\newcommand{\im}[0]{\operatorname{im}}    
\newcommand{\proj}[0]{\operatorname{Proj}}

\newcommand{\sing}[0]{\operatorname{Sing}}    
\newcommand{\ex}[0]{\operatorname{Ex}}    
  
\newcommand{\diffg}[0]{\operatorname{Diff}^*}

\newcommand{\rdown}[1]{\lfloor{#1}\rfloor}

\newcommand{\onto}[0]{\twoheadrightarrow}

\newcommand{\simq}[0]{\sim_{\q}}

\newcommand{\coeff}[0]{\operatorname{coeff}}

\newcommand{\OO}{\mathrm{O}}

\def\into{\DOTSB\lhook\joinrel\to}

\def\loccoh#1.#2.#3.#4.{H^{#1}_{#2}(#3,#4)}

\DeclareMathAlphabet{\mathchanc}{OT1}{pzc}%
                                {m}{it}

\newcommand{\api}[0]{\operatorname{\hat{\pi}_1}}

\newcommand{\DMR}{{\mathcal{DMR}}}

\newcommand{\D}{{\mathcal{D}}}

\begin{document}
\bibliographystyle{amsalpha}

\today

\title{The dual complex of Calabi--Yau pairs} 
\author{J\'anos Koll\'ar and Chenyang Xu}
\begin{abstract} A log Calabi--Yau  pair consists of a proper variety $X$
and a divisor $D$ on it such that $K_X+D$ is numerically trivial.
A folklore conjecture predicts that the dual complex of $D$ is
homeomorphic to the quotient of a sphere by a finite group. 
The main result of the paper shows that the fundamental group of the
 dual complex of $D$ is a quotient of the fundamental group of the
smooth locus of $X$, hence its pro-finite completion is finite. 
This leads to a positive answer in dimension $\leq 4$.
We also study the dual complex of degenerations of  Calabi--Yau 
varieties. The key technical result we prove is that, 
after a  volume preserving 
birational equivalence, the  transform of $D$ supports an ample divisor.
\end{abstract}

\maketitle

A {\it log Calabi--Yau  pair}, abbreviated as {\it logCY,} is a pair
$(X, \Delta)$ consisting of a proper variety $X$ and an effective
$\q$-divisor  $\Delta$ such that $(X, \Delta)$ is log canonical
and $K_X+\Delta$ is $\q$-linearly equivalent to 0.
Any  Calabi--Yau variety $X$ can be naturally 
identified with  the log Calabi--Yau pair  $(X, 0)$.
At the other extreme, if $X$ is a Fano variety and
$\Delta\simq -K_X$ is an effective divisor then
 $(X, \Delta)$ is also logCY (provided that it is log canonical).

Here we are interested in these Fano--type logCYs, 
especially when $\Delta$ is
``large.''  Being  Fano is not preserved
under birational equivalence, thus it is better to define
 ``large'' without reference to Fano varieties.
There are several natural candidates for this notion; we
were guided by the concepts of  
large complex structure limit and maximal unipotent degeneration
used in Mirror Symmetry.

\begin{defn} Let $(X, \Delta)$ be a log canonical pair of dimension $n$
and $g: (Y, \Delta_Y)\to (X, \Delta)$ a crepant log resolution.
That is, $g_*\Delta_Y=\Delta$, $K_Y+\Delta_Y\simq g^*(K_X+\Delta)$,
$X$ is smooth 
and $\ex(g)\cup\supp\Delta_Y$ is a simple normal crossing divisor.

Note that  $\Delta_Y$ is usually not effective but, since
 $(X, \Delta)$ is log canonical, 
all divisors appear in $\Delta_Y$ with coefficient $\leq 1$.
Let $\Delta_Y^{=1}$ denote the union of all irreducible components
of $\Delta_Y$ whose coefficient equals $1$.

  The combinatorics of $\Delta_Y^{=1}$ is encoded in its {\it dual complex,}
denoted by $\D\bigl(\Delta_Y^{=1}\bigr)$; see Definition \ref{dual.complex.defn}.
By \cite{dkx}   $\D\bigl(\Delta_Y^{=1}\bigr)$ is
independent of the choice of $ (Y, \Delta_Y)$, up-to 
PL-homeomorphism. We call this PL-homeomorphism type the
{\it dual complex} of $(X, \Delta)$ and
denote it by $\DMR(X, \Delta)$.

Note that  $\dim_{\r}\DMR(X, \Delta)\leq \dim X-1$ 
since, on a variety of dimension $n$, at most $n$  irreducible components
of a simple normal crossing divisor meet at a point.
We say that    $(X, \Delta)$ has {\it maximal intersection}
if equality holds.
\end{defn}

By \cite{k-dual2}, every finite simplicial complex 
of dimension $n-1$ appears
as $\D(X, \Delta)$ for some $n$-dimensional  simple normal crossing
pair $(X, \Delta)$. Thus it is
 interesting to understand which algebraic
restrictions on $(X, \Delta)$ have meaningful topological
consequences for $\DMR(X, \Delta)$. 
The  aim of this paper is to study 
the dual complex  of logCY pairs. The main result is the following.

\begin{thm}\label{main.top.thm}
 Let $(X, \Delta)$ be a logCY pair 
and $\DMR(X, \Delta)$ its dual complex. Assume that
$\dim_{\r}  \DMR(X, \Delta)\geq 2$.
Then the following hold.
\begin{enumerate}
\item  $\DMR(X, \Delta)$ 
has the same dimension at every point.
\item $H^i\bigl( \DMR(X, \Delta), \q\bigr)=0$ for
$0<i<\dim_{\r}\DMR(X, \Delta)$.
\item There is a natural surjection
$\pi_1\bigl(X^{\rm sm}\bigr)\onto \pi_1\bigl( \DMR(X, \Delta)\bigr)$.
\item The pro-finite completion $\api \bigl( \DMR(X, \Delta)\bigr)$ is finite. 
\item The cover $\widetilde{\DMR}(X, \Delta)\to \DMR(X, \Delta)$
corresponding to $\api \bigl( \DMR(X, \Delta)\bigr)$ is the
dual complex of a  quasi-\'etale cover
 $(\tilde X, \tilde \Delta)\to (X, \Delta)$.
\end{enumerate}
\end{thm}

Part (1) is a restatements of earlier results; see \cite{k-source}
or  \cite[4.40]{kk-singbook}.
The relationship between the rational homology of the dual complex and
the coherent cohomology of $X$ has been understood
for a long time; thus
 (2) has been known in many cases.  Our main
contribution is to understand the connection between the fundamental group
 of the dual complex and
the fundamental group of the smooth locus   $X^{\rm sm}\subset X$. 
The main conclusion is
(3) which in turn implies  (5).  
The finiteness of $\api\bigl(X^{\rm sm}\bigr)$ follows from \cite{MR3187625, gkp};
it is conjectured that 
$\pi_1\bigl(X^{\rm sm}\bigr)$ is finite.

If $\dim_{\r}  \DMR(X, \Delta)=1$ then $\DMR(X, \Delta)$
is either the 1-simplex $[0,1]$ or ${\mathbb S}^1$. In the latter case
$\api \bigl( \DMR(X, \Delta)\bigr)$ is infinite but usually
$X^{\rm sm}$ has only trivial quasi-\'etale covers.

On a general logCY pair, $\Delta$ contains divisors with
different coefficients, but it turns out that 
$\DMR(X, \Delta) $ is rather special if
not all coefficients are equal to 1.

\begin{prop}\label{frac.coeff.contract.prop}
 Let $(X, \Delta)$ be a logCY pair  such that
 $\Delta$ contains at least one divisor with coefficient $<1$.  Then
\begin{enumerate}
\item either $\DMR(X, \Delta) $ is contractible
\item or $\dim_{\r} \DMR(X, \Delta) \leq \dim_{\c} X-2$.
\end{enumerate}
\end{prop}

 It is straightforward to write down examples where  $\DMR(X, \Delta)$
is contractible or a sphere. More generally, at least some
quotients of spheres are easy to get; see Section  \ref{sec.examples}.
 This leads naturally to
the following.\footnote{Many people seem to have been aware of this question, among others M.~Gross, S.~Keel. V.~Shokurov,  but we could not find any specific mention in the literature.}

\begin{ques} \label{main.ques}
Let $(X, \Delta)$ be a logCY pair of dimension $n$. Is
 $\DMR(X, \Delta)\simeq {\mathbb S}^{k-1}/G$ 
for some   finite
subgroup of  $G\subset \OO_{k}(\r)$ and $k\leq n$. 
(We use $\simeq$ to denote PL-homeomorphism.)
\end{ques}

The group action may have fixed points, thus in general
$\DMR(X, \Delta)$ is only an orbifold.
LogCY varieties also appear as compactifications of character varieties.
In this context, Question \ref{main.ques}  is studied
 in \cite{2010arXiv1006.3838G, simp-2015}.
We prove the following  in Paragraph \ref{low.dim.pfs}.

\begin{prop} \label{low.dim.prop} The answer to Question \ref{main.ques}
is positive if $\dim X\leq 4$ or if $\dim X\leq 5$ and 
$(X, \Delta)$ is a simple normal crossing pair.
\end{prop}

In many questions involving Mori's program, the 3  and 4 dimensional cases
are good indicators of the general situation. 
However, the proof of Proposition \ref{low.dim.prop}
relies on several special low dimensional topological facts, thus
it gives only very weak evidence for the general problem.
We do not see any good heuristic reason why the 
answer to  Question \ref{main.ques} should be affirmative.
From the technical point of view, 
at least 3  problems remain to be settled.
\begin{itemize}
\item Finiteness of  the fundamental group of $\DMR(X, \Delta)$.
We do not know how to prove it but Theorem \ref{main.top.thm}.3
reduces it to the finiteness of
$\pi_1\bigl(X^{\rm sm}\bigr)$.  
\item Torsion in the integral homology of $\DMR(X, \Delta)$.
Our methods do not say anything about it.
\item Starting in dimension 5 we have to deal with the possibility
that  $\DMR(X, \Delta)$ is singular  but 
$X$ itself has no obvious quasi-\'etale covers.
This seems to us the most likely approach to construct a counter example
to Question \ref{main.ques}.
\end{itemize}

It is also very unclear which triangulations of a sphere can be
realized as dual complexes of logCY pairs. This is quite hard even in dimension 2; see \cite{liu-dodecahedron} for partial results.  



\begin{say}[Degenerations of Calabi--Yau varieties]\label{CY.deg.1.say}
Studying degenerations of Calabi--Yau varieties naturally leads
to log Calabi--Yau pairs. Let $g:Y\to {\mathbb D}$ be a
CY-degeneration over the unit disk ${\mathbb D}$. That is, $g$ is proper,
$K_{ Y}\simq 0$ and  $( Y, Y_0)$ is dlt
where  $Y_0$ is the central fiber.
Thus the fiber $Y_t$ is a Calabi--Yau variety for $t\neq 0$ and
the central fiber is a union of logCY pairs
$(X_i, \Delta_i)$ where $\Delta_i$ is the intersection of
 $X_i\subset Y_0$ with the other irreducible components.
We are interested in the dual complex of the central fiber
$\D(Y_0)$ especially when the following two conditions are satisfied
\begin{enumerate}
\item The general fiber $Y_t$ is
an $n$-dimensional   Calabi--Yau variety in the strict sense, that is,  
$Y_t$ is simply connected and 
$h^i(Y_t, \o_{Y_t})=0$ for $0<i<n$.
\item  $\dim\D(Y_0)=n$. 
This is a combinatorial version of the 
{\it large complex structure limit} or {\it maximal unipotent degeneration}
conditions.
\end{enumerate}
\end{say}

\begin{ques} \label{degen.ques}
Let $g:Y\to {\mathbb D}$ be a 
CY-degeneration of relative dimension $n$ 
satisfying the conditions (\ref{CY.deg.1.say}.1--2). 
Is $\D(Y_0)\simeq {\mathbb S}^n$?
\end{ques}

For $n=2$ a positive answer is given by \cite{MR0506296};
the general case is proposed in \cite{MR1882331}. 
It is easy to see that $\D(Y_0)$ is a simply connected rational homology sphere
but it is not clear that $\D(Y_0)$ is a manifold. 
Using Theorem \ref{main.top.thm},
we prove the following  in Paragraph \ref{low.dim.deg.prop.pf}.

\begin{prop} \label{low.dim.deg.prop} The answer to Question \ref{degen.ques}
is positive if $n\leq 3$ or if $n\leq 4$ and
the central fiber is a simple normal crossing divisor.
\end{prop}

\medskip

\begin{ack} We thank M.~Gross, A.~Levine, J.~Nicaise, S.~Payne, N.~Sibilla 
and C.~Simpson for
comments, discussions,  references and J.~M\textsuperscript{c}Kernan
for sharing with us early versions of  \cite{BMSZ}. We also want to thank the anonymous referee for many helpful remarks. 

Partial financial support  to JK  was provided  by  the NSF under grant number
 DMS-1362960.
Partial financial support  to CX  was provided  by   
The National Science Fund for Distinguished Young
Scholars. A large part of this work was done while CX 
enjoyed  the inspiring environment at the Institute for Advanced Studies.
\end{ack}

\section{Volume preserving maps}

For logCY pairs,
volume preserving maps, also called 
crepant birational maps, form the most important subclass of
birational equivalences.

\begin{defn}\label{crep.eq.def.1}
Let 
$(X_i, \Delta_i)$ be normal pairs. A proper, birational morphism
$g:(X_1, \Delta_1)\to (X_2, \Delta_2)$ is called 
{\it crepant} if
$g_*(\Delta_1)=\Delta_2$ and
$g^*\bigl(K_{X_2}+\Delta_2\bigr)\simq K_{X_1}+\Delta_1$.
An arbitrary birational map
$g:(X_1, \Delta_1)\map (X_2, \Delta_2)$ between proper pairs is called 
{\it  crepant} if it can be factored as
$$
\begin{array}{rcl}
& (Y,\Delta_Y) & \\
 p_1\swarrow && \searrow p_2 \\
(X_1, \Delta_1)  &\stackrel{g}{\map}&  (X_2,\Delta_2),
\end{array}
\eqno{(\ref{crep.eq.def.1}.1)}
$$
where the $p_i$ are proper, birational, crepant  morphisms. In characteristic 0
this is equivalent to having 
 a  common crepant log resolution. 
(If  the $X_i$ are not proper, the above definition still works and 
defines {\it proper and crepant} birational maps. Note that  $g$
itself is not proper, so the terminology is somewhat confusing.)

The main requirement is the  natural linear equivalence  
$$
p_1^*\bigl(K_{X_1}+\Delta_1\bigr)\simq  p_2^*\bigl(K_{X_2}+\Delta_2\bigr).
\eqno{(\ref{crep.eq.def.1}.2)}
$$
 A proper, crepant birational map
$g:(X_1, \Delta_1)\map (X_2, \Delta_2)$ between log canonical pairs is called 
{\it thrifty} if there are closed subsets
$Z_i\subset X_i$ of codimension $\geq 1$ that do not contain any of the
log canonical centers of the $(X_i, \Delta_i)$ such that
$g$ restricts to an isomorphism
$X_1\setminus Z_1\cong X_2\setminus Z_2$.

If  $(X, \Delta)$ is logCY then a global section of
$\o_X\bigl(mK_X+m\Delta\bigr)$ determines a natural volume form
on $X\setminus\supp\Delta$, up to scalar. Then a map
$g:(X_1, \Delta_1)\map (X_2, \Delta_2)$
is crepant birational iff it is {\it volume preserving}, up to a scalar.
\end{defn}

It is important to note that in (\ref{crep.eq.def.1}.1)
usually one can not choose
$Y$ such that $\Delta_Y$ is effective, not even if we allow
$Y$ to be singular.
However, as we see next, such a choice of $Y$ is possible if we allow the
$p_i$ to be rational contractions.

\begin{defn}\label{crep.eq.def.2} Let $X_i$ be normal, proper varieties. 
 A birational map $g:X_1\map X_2$ is called a
{\it  contraction} if  $g^{-1}: X_2\map X_1$
does not have any exceptional divisors.
Observe that $g$ is an {\it isomorphism in codimension 1} iff both $g$ and $g^{-1}$ are contractions.  
(Note that a birational {\em morphism} is always a contraction.
For a birational {\em map} there does not seem to be a generally
accepted notion of contraction; the above
definition is natural and quite useful.)

Let $g:(X_1, \Delta_1)\map (X_2, \Delta_2)$ 
be a birational contraction that is crepant.
Then
$g_*(\Delta_1)=\Delta_2$ and
$g^*\bigl(K_{X_2}+\Delta_2\bigr)\simq K_{X_1}+\Delta_1$.
The converse is usually not true. 
(For example,  a flop is crepant birational but a flip is not.)  
However, if the $ (X_i, \Delta_i)$ are logCY then 
the linear equivalence (\ref{crep.eq.def.1}.2)
is automatic and the converse holds. 
\end{defn}

The following lemma gives several useful ways of factoring a
crepant birational map.
 In the logCY case, a detailed  understanding of volume preserving maps
is given in \cite{cor-kal}, which also suggested to us
(\ref{factoring.crep.maps}.4). 

\begin{lem} \label{factoring.crep.maps}
Given two proper, log canonical pairs $(X_i, \Delta_i)$, 
the following are equivalent.
\begin{enumerate}
\item There is a crepant birational map
$g:(X_1, \Delta_1)\map (X_2, \Delta_2)$.
\item  There is a log canonical  pair  $(Y, \Delta_Y)$
and crepant, birational contraction maps
$$
(X_1, \Delta_1)\stackrel{p_1}{\dashleftarrow} 
(Y,\Delta_Y)\stackrel{p_2}{\map}(X_2, \Delta_2).
$$
\item There is  a pair  $(Y, \Delta_Y)$
and crepant, birational contraction maps
$$
(X_1, \Delta_1)\stackrel{p_1}{\leftarrow} 
(Y,\Delta_Y)\stackrel{p_2}{\map}(X_2, \Delta_2),
$$
where  $(Y, \Delta_Y)$ is  $\q$-factorial, dlt and $p_1$ is a morphism.
\item There are $\q$-factorial, dlt pairs  $(X'_i, \Delta'_i)$
and crepant, birational  maps
$$
(X_1, \Delta_1)\stackrel{p_1}{\leftarrow} (X'_1, \Delta'_1)
\stackrel{\phi}{\dashrightarrow}(X'_2, \Delta'_2)
\stackrel{p_2}{\to}(X_2, \Delta_2),
$$
where the $p_i$ are morphisms, $\phi$ is 
crepant birational, thrifty  and an isomorphism in codimension 1.  
\end{enumerate}
\end{lem}

Proof. It is clear that each assertion implies the previous one. 

In order to see  (1) $\Rightarrow$ (3), let $\phi:(X_1, \Delta_1){\map}(X_2, \Delta_2)$
be a crepant, birational map and
 $E_j\subset X_2$  the exceptional divisors of $\phi^{-1}$. 
The discrepancy  $a(E_j, X_1, \Delta_1)$ 
equals minus the coefficient of $E_j$ in $\Delta_2$, thus it is
 $\leq 0$. 

By  \cite[1.38]{kk-singbook} there is a 
$\q$-factorial, dlt, crepant modification 
$p_1:(Y, \Delta_Y)\to ( X_1, \Delta_1)$
that extracts precisely the $E_j$
 and possibly some other
divisors with discrepancy $-1$.
 Then $p_2:=\phi\circ p_1^{-1}$ is a 
 crepant, birational contraction map. 

The proof of (4) is similar. We take a common log resolution
$\pi_i:(Y', \Delta')\to (X_i, \Delta_i)$ and
write  $\Delta'=\Delta''+\Delta_R$ where 
$\supp \Delta''$ consists of  the set of divisors on $Y'$ that are 
in one of the following three sets:
birational transforms of $\Delta_i$, divisors that are exceptional
for exactly one of the $\pi_i$ or divisors with coefficient 1 in $\Delta'$.
Let $F$ be the sum of all other $\pi_i$-exceptional divisors.

As in  \cite[1.35]{kk-singbook},
for $i=1,2$ run the  $\bigl( Y', \Delta''+(1-\epsilon)F\bigr) $-MMP over $X_i$ 
 for   $0<\epsilon\ll 1$ to get
$p_i:(X'_i, \Delta'_i)\to (X_i, \Delta_i)$. 
If $R$ is an extremal ray that we contract then
$\bigl(R\cdot(K+\Delta''+(1-\epsilon)F)\bigr)<0$. 
Since $K+\Delta'$
is numerically $\pi_i$-trivial, this is equivalent to
$$
\bigl(R\cdot(K+\Delta''+(1-\epsilon)F-K-\Delta')\bigr)=
\bigl(R\cdot((1-\epsilon)F-\Delta_R)\bigr)<0.
$$
By our choice of $F$, $(1-\epsilon)F-\Delta_R$ is effective and 
its support is $F$. 
Thus the extremal rays
contracted are always contained in 
$F$ and the MMP contracts all the divisors contained in 
$F$.
Note also that $(Y', \Delta'') $  and $(Y', \Delta''+(1-\epsilon)F) $ 
have the same lc centers and they are not contained in $F$.
Thus $Y'\map X'_i$ is a local isomorphism  at all the lc centers of
$(Y', \Delta') $ hence the induced birational map
$(X'_1, \Delta'_1)\map (X'_2, \Delta'_2)$ is  thrifty and an isomorphism in codimension 1.
\qed
\medskip

\begin{defn}[Dual complex]\label{dual.complex.defn}
  Let $E$ be a simple normal crossing  variety over a field $k$ with irreducible components $\{E_i:i\in I\}$. 
(Our main interest is in the case $k=\c$ but  sometimes it is
convenient to allow $k$ to be arbitrary.)

A {\it stratum}
of $E$ is any irreducible component $F\subset\cap_{i\in J} E_i$ 
for some $J\subset I$.

  The  {\it dual complex} of $E$,
denoted by $\D(E)$, is a CW-complex whose vertices are
  labeled by the irreducible components of $E$ and for every stratum $F\subset
  \cap_{i\in J} E_i$ we attach a $\bigl(|J|-1\bigr)$-dimensional cell.  
Note that for
  any $j\in J$ there is a unique irreducible component of $\cap_{i\in
    J\setminus\{j\}} E_i$ that contains $F$; this specifies the attaching map.  (The  dual complex is  a regular $\Delta$-complex in
  the terminology of \cite{hatcher}.)
\end{defn}

It is very important for us that crepant birational equivalence
does not change the dual complex.

\begin{thm} \label{fkx.thm.thm}
\cite{dkx} The dual complexes of  proper, log canonical,
crepant birational pairs are PL-homeomorphic to each other.\qed
\end{thm}

Using Theorem \ref{fkx.thm.thm}
our aim is to study the dual complex $\DMR(X, \Delta)$ 
of logCY pairs
in 2 steps. First we show that  $(X, \Delta)$
is crepant birational to a ``Fano--type'' logCY pair.
There are several natural ways to define ``Fano--type.''
For our current purposes the best seems to be to
 assume that $\Delta^{=1}$ supports a
big and semi-ample divisor.

The second step is   the study of  
$\DMR(X, \Delta)$ in the Fano--type cases.

\begin{defn} Let $X$ be a variety and $D$ an effective divisor.
We say that  $D$ {\it fully supports} a divisor that is ample (or big, or semi-ample or \dots) if there is an effective divisor $H$  that is ample (or big, or semi-ample or  \dots) such that $\supp H=\supp D$.

 Example \ref{big.nef.not.rc.exmp} shows the important difference between
supporting a big and semi-ample  divisor and fully 
supporting a big and semi-ample divisor.
\end{defn}

\begin{say}[Dlt singularities] \label{delt.defn.say}
Divisorial log terminal singularities 
form a very convenient class to work with; see \cite[2.37]{km-book}
or \cite[2.8]{kk-singbook}. The precise definition
is important for the proof of Theorem \ref{p-bir}, but for
most everything else all one needs to know is that dlt pairs
behave very much like simple normal crossing pairs.

If $(X, \Delta)$ is dlt then, by \cite[Thm.3]{dkx}, 
$\DMR(X, \Delta)$ can be computed directly from 
$\Delta^{=1}$ as in Definition \ref{dual.complex.defn}.
In particular,  $\dim X=n$ and  
$(X, \Delta)$  has  maximal intersection
 iff there are $n$ divisors $D_1,\dots, D_n\subset 
\Delta^{=1}$ such that    $D_1\cap\cdots\cap D_n$
is non-empty (hence necessarily 0-dimensional). 

If $(X, \Delta)$ is dlt then for every stratum $W\subset X$
there is a natural $\q$-divisor  $ \diffg_W\Delta$ such that
$(W, \diffg_W\Delta)$ is dlt and
$
K_W+\diffg_W\Delta\simq (K_X+\Delta)|_W$;
see \cite[Sec.4.1]{kk-singbook} for the definition and basic properties.
Thus $(W,\diffg_W\Delta)$ is also logCY.

In general $ \diffg_W\Delta$ is somewhat complicated to determine,
but if all divisors in  $\Delta$ have coefficient 1 
and $K_X+\Delta$ is Cartier then
the same holds for $ \diffg_W\Delta$ and $ \diffg_W\Delta$
consists of the intersections of $W$ with those irreducible components of $\Delta$ that do not contain $W$.

This implies that  for every irreducible  divisor
$D\subset \Delta^{=1}$,
the link of $[D]\in \DMR(X, \Delta)$ is
PL-homeomorphic to $\DMR\bigl(D, \diffg_D\Delta\bigr)$.
Thus the local structure of $\DMR(X, \Delta)$ is determined
by the lower dimensional logCY pairs $\bigl(D, \diffg_D\Delta\bigr)$. 
\end{say}

\section{Reduction steps}\label{sec.reduce}

We discuss various steps that simplify $(X, \Delta)$
without changing the dual complex or changing it in 
simple ways. The final conclusion is that one should focus
on logCY pairs $(X, \Delta)$ such that
\begin{itemize}
\item  $(X, \Delta)$ is dlt, $\q$-factorial,
\item $X$ is rationally connected,
\item $\Delta=\Delta^{=1}$,
\item $\Delta$ fully   supports a big, semi-ample divisor and
\item $K_X+\Delta\sim 0$.
\end{itemize}

\begin{say}[Disconnected case] \label{disconnect.say}
By \cite{k-source} or
\cite[4.37]{kk-singbook}, if $(X, \Delta)$ is logCY and the dual complex 
$\DMR(X, \Delta)$ is disconnected 
then $(X, \Delta)$ is crepant birational to a product
$(Y,\Delta_Y)\times \bigl(\p^1, [0]+[\infty]\bigr)$
where $(Y,\Delta_Y)$ is klt. In particular
$\DMR(X, \Delta)\simeq {\mathbb S}^0$.  Thus in the sequel we need to
deal  only with the connected case.
\end{say}

\begin{say}[Index 1 cover]
Assume that  $\Delta=\Delta^{=1}$ and let $m$ be the smallest natural number
such that $m(K_X+\Delta)\sim 0$. Correspondingly there is a
degree $m$ quasi-\'etale 
(that is, \'etale in codimension 1) cover  
$(\tilde X, \tilde \Delta)\to (X, \Delta)$
and  $\DMR(X, \Delta)\simeq \DMR(\tilde X, \tilde \Delta)/\z_m$.
\end{say}

\begin{say}[Rational connectedness] \label{rc.say.1}
Let $(X, \Delta)$ be a logCY pair. If $X$ is not rationally connected then there is an MRC-fibration  $g:X\map Z$ onto a non-uniruled variety
by \cite[IV.5.2]{rc-book} and \cite{ghs}.
Let $(X_{k(Z)}, \Delta_{k(Z)})$ denote the generic fiber.

We see in  Proposition \ref{maps.onot.nonur.prop} that
 every log canonical center dominates $Z$. Thus,
by Lemma \ref{gen.fiber.lem},
$$
\DMR(X, \Delta)\simeq \DMR\bigl( X_{k(Z)}, \Delta_{k(Z)}\bigr).
$$
Thus it is enough to consider those logCY pairs whose
underlying variety is rationally connected.
\end{say}

\begin{prop} \label{maps.onot.nonur.prop}
Let $(X, \Delta) $ be a dlt logCY pair and
$g:X\map Z$ a dominant map to a non-uniruled variety $Z$.
Then every
irreducible component  $D\subset \Delta^{=1}$ dominates $Z$.
\end{prop}

Proof. 
We may assume that $Z$ is smooth and projective.
Let $C\subset X$ be a general complete intersection curve,
in particular $g$ is defined along $C$.
Then  $\bigl(C\cdot (K_X+\Delta)\bigr)=0$
and $(C\cdot D)>0$. 

If $D\subset \Delta^{=1}$ does not dominate $Z$ then
we  reach a contradiction by
proving that  $\bigl(C\cdot (K_X+\Delta)\bigr)>0$.
By blowing up $Z$ we may  assume that $D$ dominates a divisor 
 $D_Z\subset Z$;
see \cite[2.22]{kk-singbook}.

Let $Y$ be the normalization of the closure of the graph of $g$.
Choose $\Delta_Y$ such that the first projection
$(Y, \Delta_Y)\to (X, \Delta)$ is crepant.
Since $g$ is defined along $C$ we may identify $C$ with its
preimage in $Y$ and so
$\bigl(C\cdot (K_Y+\Delta_Y)\bigr)=0$.

We apply  the canonical class formula (\ref{kod.form.say}.4)
 to the second projection
$\pi_2:Y\to Z$ and write
$K_Y+ \Delta_Y\simq\pi_2^*\bigl(K_Z+J+B\bigr)$.

Note that $\bigl(C\cdot \pi_2^*K_Z\bigr)=\bigl(g(C)\cdot K_Z\bigr)\geq 0$
since $Z$ is not uniruled \cite{mi-mo} and
$\bigl(C\cdot \pi_2^*J\bigr)=\bigl(g(C)\cdot J\bigr)\geq 0$
by  (\ref{kod.form.say}.5).
Finally, (\ref{kod.form.say}.8)  shows 
that a divisor $W\subset Z$ appears in 
$B$ with  positive  (resp.\ nonnegative) coefficient if
there is divisor $W_Y\subset Y$ dominating $Y$ that   appears in 
$\Delta_Y$ with  positive  (resp.\ nonnegative) coefficient.
Thus  $g(C)$ is disjoint from the non-effective part of $B$
and intersects $D_Z$ nontrivially. Furthermore, $D_Z$  appears in $B$ with 
positive   coefficient. Thus
$\bigl(C\cdot \pi_2^*B\bigr)=\bigl(g(C)\cdot B\bigr)> 0$.
Adding these together we get that
 $\bigl(C\cdot (K_Y+\Delta_Y)\bigr)>0$, a contradiction.\qed

\begin{say}[Kodaira--type canonical class formula]\label{kod.form.say}
The original formula for elliptic surfaces was further developed by 
\cite{MR816221, MR1646046}. The following is a somewhat simplified
version of the form given in 
 \cite[8.5.1]{k-adj}. 

Let $(Y,\Delta)$ be a log canonical pair where $\Delta$ is not assumed effective.
Let $p:Y\to Z$ be a proper morphism to a normal variety $Z$
with geometrically connected generic fiber  $X_{k(Z)}$. Assume that
\begin{enumerate}
\item  $\Delta_{k(Z)}$ is effective,
\item  $\bigl(X_{k(Z)}, \Delta_{k(Z)} \bigr)$ is logCY and
\item $K_Y+ \Delta $ is $\q$-linearly equivalent to the pull-back of some $\q$-Cartier $\q$-divisor from $Z$. (This seems like a strong restriction
but it is easy to achieve by changing $\Delta$.)
\end{enumerate}
Let $Z^0\subset Z$ be the largest open set such that $p$ is
flat over $Z^0$ with logCY fibers and set $Y^0:=p^{-1}(Z^0)$. 
Then one can write
$$
K_Y+ \Delta\simq p^*\bigl(K_Z+J+B\bigr)
\eqno{(\ref{kod.form.say}.4)}
$$
where $J$ and $B$ have the following properties.
\begin{enumerate}\setcounter{enumi}{4}
\item $J$ is a $\q$-linear equivalence class, called the {\it modular part.} 
It depends only on  the generic fiber  $\bigl(X_{k(Z)}, \Delta_{k(Z)} \bigr)$ and it is 
the push-forward of a nef class by some birational morphism $Z'\to Z$.
In particular, $J$ is pseudo-effective.
\item  $B$ is a $\q$-divisor, called the {\it boundary part.}
It is  supported on $Z\setminus Z^0$. 
\item Let $D\subset Z\setminus Z^0$ be an irreducible  divisor. Then
$$
\coeff_{B}D=\sup_E\Bigl\{1-\frac{1+a(Y,\Delta, E)}{\mult_Ep^*D}\Bigr\}
$$
where the supremum is taken over all divisors over $Y$ that dominate $D$.
\end{enumerate}
This implies the following.
\begin{enumerate}\setcounter{enumi}{7}
\item If $D$ is dominated by a divisor $E$ such that
$a(Y,\Delta, E)<0$  (resp.\ $\leq 0$) then
$\coeff_{B}D>0$  (resp.\ $\geq 0$). 
\item If $\Delta$ is effective 
then so is $B$.
\item If $(Y,\Delta)$ is lc then $\coeff_{B}D\leq 1$
for every $D$ and $\coeff_{B}D=1$  iff $D$ is dominated by a divisor $E$ such that $a(Y,\Delta, E)=-1$.
\end{enumerate}
 \end{say}

The strongest reduction assertion is
the following, which is a weak form of our main technical result;
see Corollary \ref{p-bir-cor} for the general case.

\begin{thm} \label{p-bir-weak-cor}
Let $(X, \Delta)$ be a logCY pair.
Then there is a volume preserving birational map 
 $(X, \Delta)\map ( X_s, \Delta_s)$  such that
\begin{enumerate}
\item in the  maximal intersection case $\Delta_s^{=1}$ fully supports a
big and semi-ample divisor and
\item in general there is a morphism  $q_s: X_s\to Z$ 
with  generic fiber   $( X_{k(Z)},  \Delta_{k(Z)})$ such that
\begin{enumerate}
\item  $\DMR( X_{k(Z)}, \Delta_{k(Z)})\simeq\DMR( X,  \Delta)$ and
\item $ \Delta_{k(Z)} $ fully supports a
big and semi-ample divisor.
\end{enumerate}
\end{enumerate}
\end{thm}

\begin{say}[Fractional part of $\Delta$]
Assume that $\Delta^{<1}\neq 0$. Then, by \cite{bchm, MR3032329},  the $(X, \Delta^{=1})$-MMP
terminates with a Fano-contraction  $q:X_r\to Z$. 
Note that  $\DMR(X, \Delta)\simeq \DMR(X_r, \Delta_r)$ 
by Theorem \ref{fkx.thm.thm}.

If Proposition \ref{k-source.thm10.prop} applies then 
$\DMR(X_r, \Delta_r^{=1})$ is collapsible to a point and it remains to compare  $\DMR(X_r, \Delta_r)$ and
$\DMR(X_r, \Delta_r^{=1})$. In general this seems rather difficult and
it can happen that $\DMR(X_r, \Delta_r^{=1})$ is empty yet
 $\DMR(X_r, \Delta_r)$ is not.
There is, however, one case when the two are closely related.

Assume that $\Delta^{=1} $ fully supports a big and semi-ample divisor over $Z$.
 Then $\Delta_m^{=1} $  fully supports a big and mobile divisor over $Z$
for any dlt
modification
 $h:(X_m, \Delta_m)\to (X_r, \Delta_r)$.
This    implies that  $\supp \Delta_m^{=1} $ dominates $Z$ and 
$$
\supp \Delta_m^{=1} =h^{-1}\bigl(\supp \Delta_r^{=1}\bigr).
$$
Thus \cite[Thm.3]{dkx} proves that
$\DMR(X_m, \Delta_m)\simeq \DMR(X_r, \Delta_r)$ collapses to  
$\DMR(X_r, \Delta_r^{=1})$. (We do not known whether they are
PL-homeomorphic or not.)  

Thus, in this case, 
$ \DMR(X, \Delta)\simeq \DMR(X_m, \Delta_m)$ is collapsible to a point.

Combining with Theorem \ref{p-bir-weak-cor}, we have proved the following
strengthening of Proposition \ref{frac.coeff.contract.prop}.
\end{say}

\begin{cor} Let $(X,\Delta)$ be a logCY pair such that
$ \DMR(X, \Delta)$ is not collapsible to a point.

Then the model  $(X_s, \Delta_s)$ obtained in 
Theorem \ref{p-bir-weak-cor} also satisfies
$\Delta_s=\Delta_s^{=1}$ in the maximal intersection case and
$\Delta_{k(Z)}=\Delta_{k(Z)}^{=1}$ in general. \qed
\end{cor}

Let $(Y,\Delta)$ be a dlt pair and $D\subset \Delta^{=1}$ an irreducible divisor
such that  $D\cap W$ is irreducible and non-empty for every log canonical center $W$.
Then 
$\DMR(X, \Delta)$ is the cone over $\DMR\bigl(D, \diffg_{D}\Delta\bigr)$.
If   $q:(Y,\Delta)\to Z$ is a Fano contraction 
and $D$ dominates $Z$ then it  has irreducible and non-empty intersection with every log canonical center;
this is a  special case of \cite[Thm.4.40]{kk-singbook}. 
Applying these repeatedly, we obtain the following.

\begin{prop} \label{k-source.thm10.prop}
Let $(Y,\Delta)$ be a dlt pair and $q:(Y,\Delta)\to Z$ a Fano contraction. 
Assume that $\supp \Delta^{=1} $ dominates $Z$.
Then there is  a unique smallest lc center $W\subset Y$ dominating $Z$ and
$$
\DMR(X, \Delta)=\sigma^{r}* \DMR\bigl(W, \diffg_{W}\Delta\bigr),
$$ 
the join of a simplex $\sigma^{r}$ of dimension $r=\codim_YW-1$ and of
$\DMR\bigl(W, \diffg_{W}\Delta\bigr)$. 

In particular,  $\DMR(X, \Delta)$ is
contractible, even  collapsible.
\qed
\end{prop}

Note also that the simplex $\sigma^r$ is PL-homeomorphic to
${\mathbb S}^r/(\tau)$ where $\tau$ is a reflection on a hyperplane.
Thus, as in Example \ref{join.and.product.exmp},
if $\DMR\bigl(W, \diffg_{W}\Delta\bigr) $ is the quotient of a sphere
then so is $ \DMR(Y, \Delta)$.

\section{Basic results on the dual complex}

We need two results that connect the topology of $\D(E)$
and the algebraic geometry of $E$.
The following homological lemma is essentially proved in 
\cite[pp.68--72]{gri-sch}; see also  \cite[pp.26--27]{friedman-etal}
and \cite[3.63]{kk-singbook}. The fundamental group
result is rather straightforward; \cite[Lem.25]{k-fg}.

\begin{lem}\label{friedman??.lem}
  Let $E=\cup_{i\in I}E_i$ be a proper, simple normal crossing  variety over $\c$. 
Then there are natural injections
$$
H^r\bigl(\D(E), \c\bigr)\into H^r\bigl(E,  \o_{E}\bigr).
\eqno{(\ref{friedman??.lem}.1)}
$$
Furthermore, if  $H^r\bigl(\cap_{i\in J}E_i, \o_{\cap_{i\in J}E_i}\bigr)=0$ for every $r>0$ and  every $J\subset I$ then (\ref{friedman??.lem}.1) is an isomorphism.\qed
 \end{lem}

\begin{lem}\label{pi1.pf.D.lem}  Let $E=\cup_{i\in I}E_i$ be a connected, simple normal crossing variety over $\c$.
 Then there is a natural surjection 
$$
\pi_1(E)\onto \pi_1\bigl(\D(E)\bigr).
\eqno{(\ref{pi1.pf.D.lem}.1)}
$$
Furthermore, if the  $E_i$ are simply connected then
(\ref{pi1.pf.D.lem}.1) is an isomorphism.\qed
\end{lem}

In some cases one can describe a dual complex using a fibration
and finite group actions as in Theorem \ref{p-bir-weak-cor}.2.

\begin{say} \label{top.of.quot.lem}
Let $T$ be a simplicial complex and $G$ a finite group acting
on $T$. Let  $B(T)$ denote the barycentric subdivision of $T$.
If $C\in B(T)$ is a simplex and $g\in G$ such that $g(C)=C$
then $g$ also fixes every vertex of $C$. Thus the $G$ action on
$B(T)$ naturally extends to a simplicial $G$-action on the
topological realization $|B(T)|$. 

The quotient is a regular  complex denoted by $B(T)/G$.
There is a natural map  $|B(T)|\to |B(T)/G|$ whose fibers
are exactly the $G$-orbits on $|B(T)|$.

Such  branched covering spaces
are discussed in  \cite{MR0123298}; we need the following
  properties.

\begin{enumerate}
\item There are natural isomorphisms
$H^i\bigl(|T|/G, \q\bigr)\cong H^i\bigl(|T|, \q\bigr)^G$ for every $i$.
\item There is an exact sequence 
 $$
\pi_1\bigl(|T|\bigr)\to \pi_1\bigl(|T|/G\bigr)\to 
G/\langle \mbox{stabilizers of points}\rangle.
$$
In particular, if $ \pi_1\bigl(|T|\bigr)$ is finite then so is
$ \pi_1\bigl(|T|/G\bigr)$.
\end{enumerate}
\end{say}

The quotient construction naturally arises for
families of simple normal crossing varieties.

\begin{lem} \label{D.from.fiber.lem}
$E=\cup_{i\in I}E_i$ be a  simple normal crossing variety
 and $q:E\to Z$ a morphism
such that every stratum of $E$ dominates $Z$. Let $z\in Z$ be a general point
and $E_z$ the fiber over $z$. Then $E_z$ is a simple normal crossing variety
and there is a finite group $G$ acting
on $\D(E_z)$ such that  $\D(E)=\D(E_z)/G$.
\end{lem}

Proof. By shrinking $Z$ we may assume that $q$ is smooth on every stratum
of $E$. Then $z\mapsto \D(E_z)$ defines a locally trivial fiber bundle.
Since $\D(E_z)$ is a finite simplicial complex, the monodromy of
the fiber bundle is a finite group $G$ and $\D(E)=\D(E_z)/G$.\qed

\medskip

Algebraically minded readers may prefer to think of
Lemma \ref{D.from.fiber.lem} as a combination of the next
two claims.

\begin{lem} \label{gen.fiber.lem}
$E=\cup_{i\in I}E_i$ be a  simple normal crossing variety over a field $k$
 and $q:E\to Z$ a morphism.
Then the generic fiber   $E_{k(Z)}$ is a 
simple normal crossing variety over the function field $k(Z)$ and
$\D(E_{k(Z)}) $ is a subcomplex of $\D(E) $.
Furthermore, if every stratum dominates 
$Z$ then 
$\D(E_{k(Z)})=\D(E)$. \qed
\end{lem}

\begin{lem} \label{gal.ext.lem}
Let $K/k$ be a Galois extension with Galois group $G$.
Let $E_k$ be a  simple normal crossing variety over  $k$.
 Then $G$ acts on
$\D(E_K)$ and
$\D(E_k)=\D(E_K)/G$. \qed
\end{lem}

\section{Homology of  the dual complex}

The following proves  (\ref{main.top.thm}.2).

\begin{prop} \label{dlt.logcy.RHS.cor} Let $(X, \Delta)$ be a  logCY pair. 
Then 
$$
H^i\bigl(\DMR(X,\Delta), \q\bigr)=0\qtq{for} 0<i<\dim \DMR(X,\Delta).
$$
\end{prop}

Proof.  
Assume first that $X$ is rationally
connected, $\Delta=\Delta^{=1}$ and $ K_X+\Delta^{=1}\sim 0$. 
Set $n:=\dim X$. Then 
$H^i(X, \o_X)=0$ for $i> 0$ by \cite{MR1094468, KMM92a} and 
$$
H^i\bigl(X, \o_X(-\Delta^{=1})\bigr)=H^i\bigl(X, \o_X(K_X)\bigr)=0
\qtq{for} i< n
$$
 by Serre duality.
The long cohomology sequence of the exact sequence
$$
0\to \o_X(-\Delta^{=1})\to \o_X\to \o_{\Delta^{=1}}\to 0
\eqno{(\ref{dlt.logcy.RHS.cor}.1)}
$$
now implies that $H^i\bigl(\Delta^{=1}, \o_{\Delta^{=1}}\bigr)=0$
for $0<i<n-1$. Thus
$H^i\bigl(\DMR(X,\Delta), \q\bigr)=0 $
for $0<i<n-1$ by Lemma \ref{friedman??.lem}.

We try to use a similar argument in general; the problem is that,
in (\ref{dlt.logcy.RHS.cor}.1),
instead of $\o_X(K_X)$ we have 
$\o_X(-\Delta^{=1})$ and $-\Delta^{=1}$ is $\q$-linearly
equivalent to $K_X+\Delta^{<1}$. The presence of the
fractional part $\Delta^{<1}$ and the $\q$-linear
(as opposed to linear) equivalence both cause problems.

Assume next that $(X, \Delta)$  is dlt and 
$\Delta^{=1}$ fully supports a big and semi-ample
divisor $M$. Note that
$$
0\simq K_X+\epsilon M+ \bigl(\Delta-\epsilon M)
\qtq{and}
\Delta^{=1}\sim \epsilon M + \bigl(\Delta^{=1}-\epsilon M),
$$
thus by vanishing we see that
$$
H^i(X, \o_X)=0\qtq{for $i> 0$ and}
H^i\bigl(X, \o_X(-\Delta^{=1})\bigr)=0\qtq{for $i< n$.}
$$
Using (\ref{dlt.logcy.RHS.cor}.1) these imply that 
$H^i\bigl(\Delta^{=1}, \o_{\Delta^{=1}}\bigr)=0$
for $0<i<n-1$.

Finally, by Theorem \ref{fkx.thm.thm}, we are free to replace
$(X, \Delta)$ with any other logCY pair that is crepant
birational to it.
We first  apply Theorem \ref{p-bir-weak-cor}, then  
  Lemmas \ref{D.from.fiber.lem} and \ref{top.of.quot.lem}
to reduce to the already established case when
$\Delta^{=1}$ fully supports a big and semi-ample
divisor.\qed 
\medskip

Next we study the top cohomology of
the structure sheaf and of the dual complex.

\begin{say}[The top cohomology of  logCY pairs]\label{top.coh.say}
Let $(X, \Delta)$ be a dlt, logCY pair of dimension $n$.
Then $H^n(X, \o_X)$ is  Serre dual to $H^0\bigl(X, \o_X(K_X)\bigr)$
hence $H^n(X, \o_X)=0$ save when $\Delta=0$ and $K_X\sim 0$. 

Let  $D\subset \Delta^{=1}$ be an irreducible component such that
$H^{n-1}(D, \o_D)\neq 0$. As we noted above, then $\diffg_D\Delta=0$,
thus $D$ is a connected component
of $\supp \Delta$. As we noted in Paragraph \ref{disconnect.say},  
there are 2 possibilities.
Either $ \Delta^{=1}=D$ or $\Delta^{=1}=D+D'$ has 2 irreducible components
 which are crepant birational to each other. 

Let $(Y, \Delta_Y)$ be a dlt pair.
Set $X:= \Delta_Y^{=1}$ and assume that   $\bigl(C\cdot (K_Y+\Delta_Y)\bigr)=0$
for every curve $C\subset X$.  We can then view
$(X,\Delta:=\diffg_X\Delta_Y)$ as a reducible logCY pair. 
A non-embedded definition of such pairs, called {\it semi-dlt} pairs, is given in
\cite[Sec.5.4]{kk-singbook}. The precise definition is not important for now,
we will only use the case when $X= \Delta_Y^{=1}$ as above.

Using these observations inductively, we get the following.
\medskip

{\it Claim \ref{top.coh.say}.1.} 
 Let $(X, \Delta)=\cup_i (X_i, \Delta_i)$ be a connected, semi-dlt, 
logCY pair of dimension $n$. Assume that it has a stratum
$W\subset X$ of dimension $r$ such that $H^r(W, \o_W)\neq 0$. 
Then  
\begin{enumerate}
\item all strata have dimension $\geq r$,
\item the  $r$-dimensional  strata are crepant birational to each other and
\item $\dim\D(X)= n-r-1$. \qed
\end{enumerate}
\medskip

Let $X=\cup_{i\in I}X_i$ be a simple normal crossing variety.
The cohomology of $\o_X$ is computed by a spectral sequence
whose $E_1$ terms are
$$
E_1^{pq}=H^q\bigl(X_p, \o_{X_p}\bigr)\qtq{where}
X_p:=\coprod_{J\subset I, |J|=p+1}\ \cap_{i\in J} X_i.
\eqno{(\ref{top.coh.say}.2)}
$$
In the bottom row $q=0$ we find the complex that computes the
cohomology of $\D(X)$. Note also that if $H^{\dim W}(W, \o_W)= 0$
holds for every positive dimensional stratum then
the only term that contributes to $H^{n}(X, \o_X)$ 
is $E_2^{n0}=H^n\bigl(\D(X), \c\bigr)$. We have thus proved the following.

\medskip

{\it Claim \ref{top.coh.say}.3.} 
 Let $(X, \Delta)=\cup_i (X_i, \Delta_i)$ be a connected, semi-dlt, 
logCY pair of dimension $n$ such that $H^{n}(X, \o_X)\neq 0$.
Then \begin{enumerate}
\item either  $H^n\bigl(\D(X), \q\bigr)=\q$, 
\item or $\dim\D(X)<n$.
\qed
\end{enumerate}

\end{say}

Next we prove Proposition \ref{low.dim.prop}.

\begin{say}[Dimension induction] \label{low.dim.pfs}
Let $(X, \Delta)$ be a dlt logCY pair of dimension $n$ 
 such that $K_X+\Delta\sim 0$.  
As we noted in Paragraph \ref{delt.defn.say},
 the local structure of $\DMR(X, \Delta)$ is determined
by the lower dimensional logCY pairs. 

Assume first that $(X, \Delta)$ has maximal intersection.
Using Proposition \ref{dlt.logcy.RHS.cor} and Paragraph \ref{top.coh.say} we see that 
$H^{n-1}\bigl(\DMR(X, \Delta), \q\bigr)=\q$ hence
$\DMR(X, \Delta)$ is a  rational homology sphere. 
In low dimensions we obtain complete answers to  Question \ref{main.ques}.

If $n=1$ then $X=\p^1$ and $\DMR(X, \Delta)\simeq {\mathbb S}^0$.

If $n=2$ then $X$ is a 
rational surface and $\DMR(X, \Delta)\simeq {\mathbb S}^1$.

If $n=3$ then $\DMR(X, \Delta)$ is a  2--manifold 
that is a rational homology sphere. 
Thus $\DMR(X, \Delta)\simeq {\mathbb S}^2$.

If $n=4$ then $\DMR(X, \Delta)$ is a  3--manifold 
that is a rational homology sphere. 
The fundamental group of a 3--manifold is residually finite
\cite{MR895623}, thus (\ref{main.top.thm}.4) implies that
the fundamental group itself is finite. By (\ref{main.top.thm}.5)
the universal cover is a simply connected homology sphere,
thus $\widetilde{\DMR}(X, \Delta)\simeq {\mathbb S}^3$.
(This uses the Poincar\'e conjecture.)
Note however that we do not claim that 
$\DMR(X, \Delta)$ is the sphere $ {\mathbb S}^3$.

Thus, starting with $n=5$ we do not claim that $\DMR(X, \Delta)$
is a manifold.

We can, however,  do better if $(X, \Delta)$ is a simple normal crossing pair.
In this case Theorem \ref{pi1Xns.to.pi1D.thm.mi}
implies that $\DMR(X, \Delta)$ and its links are simply connected.
Thus we get that $\DMR(X, \Delta)$ is homeomorphic to a sphere 
if $\dim X\leq 5$. (Conjecturally, PL-homeomorphic to a sphere.)

If $n=6$ then $\DMR(X, \Delta)$ is a  5--manifold 
that is a simply connected rational homology sphere. 
Our results say nothing about the torsion group
$H_2\bigl(\DMR(X, \Delta), \z\bigr)$.

Finally consider the case when $(X, \Delta)$ 
does not have  maximal intersection. A similar induction shows that,
for $\dim \DMR(X, \Delta)\leq 3$, the universal cover is
a simply connected manifold, possibly with boundary. Thus  
$\DMR(X, \Delta)$ is either a sphere or a ball.
\end{say}

\begin{say}\label{low.dim.deg.prop.pf}
Let $g:Y\to {\mathbb D}$ be a 
CY-degeneration of relative dimension $n$ 
satisfying the conditions (\ref{CY.deg.1.say}.1--2). 

If $X_i\subset Y_0$ is an irreducible component 
and $D_i\subset X_i$ is the intersection of $X_i$ with the other components
then $(X_i, D_i)$ is a logCY pair and
$\DMR(X_i, D_i)$ is PL-homeomorphic to the link of
$[X_i]\in \D(Y_0)$. Thus if $n\leq 3$ or if $n\leq 5$ and
the central fiber is a simple normal crossing divisor
then  $\D(Y_0)$ is a manifold using the results of Paragraph
\ref{low.dim.pfs}. 

Since $Y_0$ has Du~Bois singularities  (cf.\ \cite[Chap.6]{kk-singbook})
we see that $h^i(Y_0, \o_{Y_0})=h^i(Y_t, \o_{Y_t})=0$ for
$0<i<n$, thus $\D(Y_0)$ is a rational homology sphere.

Finally, there are natural surjections
$$
\pi_1\bigl(Y\setminus Y_0\bigr)\onto \pi_1(Y)\cong \pi_1(Y_0)
\onto \pi_1\bigl(\D(Y_0)\bigr).
$$
We assumed that $Y_t$ is simply connected for $t\neq 0$, hence
$\pi_1\bigl(Y\setminus Y_0\bigr) $ is a quotient of
$\pi_1\bigl({\mathbb D}\setminus\{0\}\bigr)\cong \z$.
Since $Y_0$ is reduced, $Y\to {\mathbb D}$ has  a section,
thus $\pi_1\bigl({\mathbb D}\setminus\{0\}\bigr)$  gets killed
in $\pi_1(Y_0) $.

These imply that $\pi_1\bigl(\D(Y_0)\bigr) $ is trivial
and so $\D(Y_0)$ is a simply connected rational homology sphere.
For $n\leq 4$ this implies that it is 
homeomorphic to a sphere
 (conjecturally PL-homeomorphic to a sphere).
This completes the proof of Proposition 
\ref{low.dim.deg.prop}.
\end{say}

\section{Fundamental groups of logCY pairs}

In this section we study various fundamental groups associated to a
logCY pair. It is easy to see that usually $X$ itself is simply connected.
It is much more interesting to understand the  fundamental group
of the smooth locus $\pi_1\bigl(X^{\rm sm}\bigr) $
and the fundamental group of the dual complex.
Note that while the latter is a crepant-birational invariant,
the fundamental group
of the smooth locus is not; see Example \ref{pi1.not.birinv.exmp}.

\begin{say}[General set-up] \label{pi1.setup.say}
Let $X\subset \c\p^N$ be a normal, projective variety
and $D\subset X$ a divisor.  Fix a smooth metric on $\c\p^N$ and 
for $0<\delta\ll 1$ let $D_{\delta}\subset X$ denote the
$\delta$-neighborhood of $D$.  Then $D$ is a deformation retract
of $D_{\delta}$, hence $ \pi_1(D_{\delta})\cong \pi_1(D)$. 
If $(X,D)$ is snc or  dlt then we can form the dual complex $\D(D)$
and we have a surjection
$$
\pi_1(D_{\delta})\cong \pi_1(D)\onto  \pi_1\bigl(\D(D)\bigr).
$$
If $Z\subset X\setminus D_{\delta}$ is any closed subset then there
is a natural map $L_D: \pi_1(D_{\delta})\to \pi_1(X\setminus Z)$.
Very little is known about this map in general but if $D$ is ample
and $\dim X\geq 3$ then it is an isomorphism 
by the Lefschetz hyperplane theorem for large enough finite $Z\subset \sing X$. Thus we have the following inclusions
$$
X^{\rm sm}\into X\setminus Z \hookleftarrow D_{\delta} \hookleftarrow D
$$
and these induce maps on the fundamental groups
$$
\pi_1\bigl(X^{\rm sm}\bigr)\onto \pi_1(X\setminus Z)
\stackrel{L_D}{\longleftarrow} \pi_1(D_{\delta})\cong
\pi_1(D)\onto \pi_1\bigl(\D(D)\bigr).
\eqno{(\ref{pi1.setup.say}.1)}
$$
Thus if $L_D$ is an isomorphism then we obtain surjections
$$
\pi_1\bigl(X^{\rm sm}\bigr) \onto \pi_1(X\setminus Z)\onto
\pi_1(D)\onto \pi_1\bigl(\D(D)\bigr).
\eqno{(\ref{pi1.setup.say}.2)}
$$
We would like to apply this to $(X, \Delta^{=1})$ for a 
 logCY pair  $(X, \Delta)$. Typically $\Delta^{=1} $ is quite negative
but Theorem \ref{p-bir} shows that ampleness can be achieved for a
suitable crepant birational model  $(\bar X, \bar\Delta)$. Unfortunately
$(\bar X, \bar\Delta) $  is not dlt. We thus have to find a dlt
model $(X_s, \Delta_s)$ where $\Delta_s^{=1} $ is close enough to
being ample that the above arguments apply. An extra complication is
that $ \pi_1\bigl(X^{\rm sm}\bigr)$ is not a birational invariant.
At the end we prove (\ref{pi1.setup.say}.2) for $D=\Delta^{=1} $.

If (\ref{pi1.setup.say}.2) holds then any finite covering space
  $\widetilde{\D(D)}\to \D(D)$ lifts to a finite,  \'etale
cover of $X\setminus Z$, hence to a finite, possibly ramified  cover  
$p: \tilde X\to X$ that is \'etale along $D$. 
If every prime divisor $E\subset X$
has non-empty intersection with $D$ then $p$ is 
quasi-\'etale.
If $(X, D)$ is dlt, logCY and  $p$ is 
quasi-\'etale then   $\bigl(\tilde X, \tilde D:=p^{-1}(D)\bigr)$
is also logCY
and $\D\bigl(\tilde D\bigr)=\widetilde{\D(D)}$.

This gives a realization of covering spaces of $\D(D)$ 
as  dual complexes of quasi-\'etale covers of $(X,D)$.  
Thus (\ref{pi1.setup.say}.2) for   $D=\Delta^{=1} $  implies (\ref{main.top.thm}.5).
\end{say}

The main result of this section is the following restatement of (\ref{main.top.thm}.3).

\begin{thm} \label{pi1Xns.to.pi1D.thm.mi}
Let $(X, \Delta)$ be a dlt logCY pair 
such that $\dim \DMR(X,\Delta)\geq 2$.
Then there  are surjections
$$
\pi_1\bigl(X^{\rm sm}\bigr)\onto \pi_1(\Delta^{=1})\onto \pi_1\bigl(\DMR(X,\Delta)\bigr).
$$
\end{thm}

We start with the equality of the last 2 groups.

\begin{lem} \label{dlt.logcy.rc.cor} Let $(X, \Delta)$ be a dlt logCY pair that has maximal intersection.
Then 
$ \pi_1(\Delta^{=1})\cong  \pi_1\bigl(\DMR(X,\Delta)\bigr)$.
\end{lem}

Proof.  By \cite[4.40]{kk-singbook}
the irreducible components of $\Delta^{=1} $ are also dlt logCY pairs that have maximal intersection. Thus
Proposition   \ref{maps.onot.nonur.prop}            
implies that they are rationally connected.
A proper, rationally connected variety is simply connected by
 \cite{MR1094468, KMM92a},
thus  Lemma \ref{pi1.pf.D.lem} shows that
$$
 \pi_1(\Delta^{=1})\cong  \pi_1\bigl(\D(\Delta^{=1})\bigr)\cong  \pi_1\bigl(\DMR(X,\Delta)\bigr). \qed
$$

We next prove a variant of Theorem 
\ref{pi1Xns.to.pi1D.thm.mi}.

\begin{prop} \label{pi1Xns.to.pi1D.thm.sa}
Let $(X, \Delta)$ be a dlt logCY pair of dimension $\geq 3$.
Assume that  $\Delta^{=1}$ fully supports a big and semi-ample divisor
and every prime divisor $E\subset X$
has non-empty intersection with  $\Delta^{=1}$.
Then there  are surjections
$$
\pi_1\bigl(X^{\rm sm}\bigr)\onto \pi_1(\Delta^{=1})\onto \pi_1\bigl(\DMR(X,\Delta)\bigr).
$$
\end{prop}

\begin{proof} By assumption there is a  birational morphism $f:X\to Y$ 
and an ample divisor $H$ on $Y$ such that $\supp f^{-1}(H)=\supp \Delta^{=1}$.
Corollary \ref{l-lef} shows that  
$\pi_1\bigl(X^{\rm sm}\bigr)\onto \pi_1(\Delta^{=1}) $
is surjective while $\pi_1(\Delta^{=1})\onto \pi_1\bigl(\DMR(X,\Delta)\bigr) $  
is surjective by Lemma \ref{pi1.pf.D.lem}.
\end{proof}

The key ingredient of the above proof is the following immediate consequence of \cite[Part II, Thm.1.1]{gm-book}.
As in Paragraph \ref{pi1.setup.say},  $H_{\delta}\supset H$
denotes a   neighborhood of $H$
(in the Euclidean topology) such that $H$ is a deformation retract of 
$H_{\delta} $. 

\begin{prop}\label{p-Lef}
Let $U$ be a smooth (possibly non-proper) variety,    $f:U\to W$  a birational morphism and  $H$ an ample divisor on $W$.  Assume that $\dim f^{-1}(w)\leq \dim U-2$ for every $w\in W\setminus H$.  
 Then there is a natural isomorphism
$$\pi_1\bigl(f^{-1}H_{\delta}\bigr)\cong \pi_1(U).\qed
$$
\end{prop}

 \begin{cor}\label{l-lef}
 Let $f:X\to Y$ be a  birational morphism between normal, projective varieties. 
Let $H$ be an ample divisor on $Y$ such that every prime divisor $E\subset X$
has non-empty intersection with $f^{-1}(H)$.
Then there is a natural surjection 
 $$\pi_1(X^{\rm sm})\onto \pi_1(f^{-1}(H)). $$ 
 \end{cor}

Proof.
Set $H^{\rm sm}_{\delta}:=f^{-1}H_{\delta} \cap X^{\rm sm} $. By Proposition \ref{p-Lef}, 
there is an isomorphism
$$\pi_1 (H^{\rm sm}_{\delta})\cong \pi_1(X^{\rm sm})$$
and  the natural injection 
$H^{\rm sm}_{\delta}\into f^{-1}(H_{\delta})$ induces a surjective morphism
$$\pi_1(H^{\rm sm}_{\delta})\onto \pi_1(f^{-1}(H_{\delta})).$$

Finally we can choose $H_{\delta} $ such that $f^{-1}(H_{\delta})$ retracts to   $f^{-1}(H)$ hence
$\pi_1(f^{-1}(H_{\delta}))\cong \pi_1(f^{-1}(H))$. Combining these
we get a surjection 
 $$\pi_1(X^{\rm sm})\onto \pi_1(f^{-1}(H)).\qed
 $$

\begin{lem}\label{l-pure}
Let $f:X\map Y$ be a rational contraction between normal, proper varieties.
Then there is a  natural surjection
$$
\pi_1\bigl(Y^{\rm sm}\bigr)\onto \pi_1\bigl(X^{\rm sm}\bigr).
$$
\end{lem}

\begin{proof} $f^{-1}$ gives a birational morphism
$Y^{\rm sm}\setminus \ex(f^{-1})\into X^{\rm sm}$
which gives a surjection
$$
\pi_1\bigl(Y^{\rm sm}\setminus \ex(f^{-1})\bigr)\onto \pi_1\bigl(X^{\rm sm}\bigr).
$$
Since $f$ is a rational contraction, 
$ \ex(f^{-1})$ has complex codimension $\geq 2$ thus
$\pi_1\bigl(Y^{\rm sm}\setminus \ex(f^{-1})\bigr)=
\pi_1\bigl(Y^{\rm sm}\bigr)$.
\end{proof}

\begin{say}[Proof of Theorem  \ref{pi1Xns.to.pi1D.thm.mi}]
\label{pf.pf.pi1Xns.to.pi1D.thm.mi}
The plan is to use Theorem \ref{p-bir} to find a crepant birational
logCY pair $(\bar X, \bar \Delta)$ such that
$\bar\Delta^{=1}$ fully supports an ample divisor $H$ and then
use Proposition \ref{pi1Xns.to.pi1D.thm.sa} on a dlt modification
of $(\bar X, \bar \Delta)$.
This almost works and there are  only two problems: 
$\pi_1(X^{\rm sm})$ is not a birational invariant and
we have to deal with the relative case.
In order to deal with the first problem
 we have to work with a different model to which
Proposition \ref{pi1Xns.to.pi1D.thm.sa} does not apply directly.
Nonetheless we follow this path and use
 \cite[Thm.3]{dkx} to go around the difficulties.

 We use Corollary \ref{p-bir-cor} to obtain 
 crepant birational maps
$$
\psi:=g^{-1}\circ \phi: (X, \Delta)\stackrel{\phi}{\map} (\bar X, \bar \Delta)
\stackrel{g}{\leftarrow} (X_s, \Delta_s) 
\eqno{(\ref{pf.pf.pi1Xns.to.pi1D.thm.mi}.1)}
$$
satisfying the properties (\ref{p-bir-cor}.1--5). 
By Lemma \ref{l-pure} we have a   natural surjection
$\pi_1\bigl(X^{\rm sm}\bigr)\onto \pi_1\bigl(X_s^{\rm sm}\bigr)$
and $\DMR(X,\Delta)\simeq \DMR(X_s,\Delta_s)$.
Using  \cite[Thm.3]{dkx} we obtain  that $ \D(g^{-1}(H))$ is
homotopy equivalent to $\DMR(X_s,\Delta_s) $,
hence 
$$
\pi_1\bigl(\D(g^{-1}(H))\bigr)\cong
\pi_1\bigl(\DMR(X_s,\Delta_s)\bigr)\cong
\pi_1\bigl(\DMR(X,\Delta)\bigr).
\eqno{(\ref{pf.pf.pi1Xns.to.pi1D.thm.mi}.2)}
$$
Thus it is enough to prove that  there is a natural surjection 
 $$
\pi_1(X_s^{\rm sm})\onto \pi_1(g^{-1}(H)).
\eqno{(\ref{pf.pf.pi1Xns.to.pi1D.thm.mi}.3)}
$$
Next we look at the morphism $q_s:X_s\to Z$ in Corollary \ref{p-bir-cor}.
If $Z$ is a point  (this is always the case  if $(X, \Delta)$ has maximal intersection) the surjectivity of (\ref{pf.pf.pi1Xns.to.pi1D.thm.mi}.3) is implied by 
 Corollary \ref{l-lef}.

If $Z$ is positive dimensional, let $Z^*\subset Z$ be an open subset
such that both $q_s:X_s\to Z$ and its restriction $q_H:g^{-1}(H)\to Z$
become  topological fiber bundles 
$q_s^*:X_s^*\to Z^*$ and  $q_H^*:g^{-1}(H)^*\to Z^*$.
For both of these, the fundamental group surjects onto the
fundamental group of $Z^*$ and the kernel is the image of the 
fundamental group of the fiber. We have already established a
surjection between the fundamental groups of the respective fibers,
thus we have a surjection
$$
\pi_1\bigl(X_s^*\bigr)\onto  \pi_1\bigl(g^{-1}(H)^*\bigr),
\eqno{(\ref{pf.pf.pi1Xns.to.pi1D.thm.mi}.4)}
$$
which in turn gives a surjection
$$
\pi_1\bigl(X_s^*\bigr)\onto  \pi_1\bigl(g^{-1}(H)\bigr).
\eqno{(\ref{pf.pf.pi1Xns.to.pi1D.thm.mi}.5)}
$$
Equivalently,  every \'etale cover of $g^{-1}(H) $ comes from an
\'etale cover of a suitable Zariski open neighborhood  $X_s\supset U\supset g^{-1}(H) $. 
Now we can use    (\ref{p-bir-cor}.4) to obtain that
$X_s\setminus U$ has codimension $\geq 2$, thus, by the purity of branch loci, 
(\ref{pf.pf.pi1Xns.to.pi1D.thm.mi}.5)   
extends to a surjection
$$
\pi_1\bigl(X_s^{\rm sm}\bigr)\onto  \pi_1\bigl(g^{-1}(H)\bigr).
\eqno{(\ref{pf.pf.pi1Xns.to.pi1D.thm.mi}.6)}
$$
Putting all of these together we have a chain
of natural surjections
$$
\pi_1(X^{\rm sm})\onto\pi_1(X_s^{\rm sm})\onto \pi_1(g^{-1}(H))   \onto \pi_1\bigl(\D(g^{-1}(H))\bigr)\cong
\pi_1\bigl(\DMR(X_s,\Delta_s)\bigr). \qed
$$
\end{say}

The arguments of Paragraph \ref{pi1.setup.say} show the 
following  variant of  (\ref{main.top.thm}.5).

\begin{cor} \label{all.covers.geom.cor}
Let $(X, \Delta)$ be a logCY pair that has maximal intersection. Then every 
finite degree covering space of
$\DMR(X,\Delta) $ is the dual complex of a 
suitable quasi-\'etale cover of $(X, \Delta)$.\qed
\end{cor}

It is natural to conjecture that, under the assumptions
of Theorem \ref{pi1Xns.to.pi1D.thm.mi},
 $\pi_1\bigl(X^{\rm sm}\bigr) $---and hence also the other 2 groups---are finite,
 but we have only the following weaker result.

\begin{prop} \label{Xsm.sfg.fin.cor}
Let $(X, \Delta)$ be a dlt, logCY pair. Then
\begin{enumerate}
\item either  $\api(X^{\rm sm})$ is finite
\item or there is a quasi-\'etale cover 
$(\tilde X, \tilde\Delta)\to (X, \Delta)$ 
and a dominant map $q:\tilde X\map Z$ onto a non-uniruled variety.
\end{enumerate}
In particular, if $(X, \Delta)$ has maximal intersection
then  the groups $\api(X^{\rm sm})$ and  $\api\bigl(\DMR(X,\Delta)\bigr) $ are finite. 
\end{prop}

The key result we use is the following local variant.

\begin{thm}\cite{MR3187625}\label{MR3187625.thm}
Let $(0\in X)$  be a normal singularity 
over $\c$  that is
potentially dlt. (That is, there is a divisor $\Delta$ such that
$(X,\Delta)$ is dlt.) Then there is a Euclidean open neighborhood
$0\in V\subset X$ such that $\api(V\setminus\{0\})$ is finite.
 \qed
\end{thm}

A formal argument can be used to go from the local to the global
situation.

\begin{say}[Local and global fundamental groups]\label{lgfg.say}  
Let $M$ be a
compact simplicial complex and
$C\subset M$ a  closed subcomplex.  Set $N:=M\setminus C$ and let
$B\subset N$ be a  nowhere dense, closed subcomplex. Assume that
 every $p\in B$
has a contractible neighborhood $U_p$ such that 
$V_p:=U_p\cap (N\setminus B)$ is connected.

By repeatedly applying van Kampen's theorem we get the following.
\medskip

{\it Claim \ref{lgfg.say}.1.} The natural map 
$\pi_1(N\setminus B)\to \pi_1(N)$  is surjective and its kernel
is generated (as a normal subgroup)
 by the images of the maps  $\pi_1(V_p)\to \pi_1(N\setminus B)$.
The same holds for the pro-finite completion
$\api(N\setminus B)\to \api(N)$. \qed
\medskip

 Assume now that the $\api (V_p)$  are all finite.
Since $M$ is compact,  the images  $\im[\api(V_p)\to \api(N\setminus B)]$
form finitely many conjugacy classes.
Thus there is a finite index normal subgroup
 $G\subset\api (N\setminus B)$ 
such that 
$$
G\cap \im[\api(V_p)\to \api(N\setminus B)]=\{1\} \quad \forall p\in B.
\eqno{(\ref{lgfg.say}.2)}
$$
Let  $ s_B:  (N\setminus B)^{\sim}\to (N\setminus B)$
be the corresponding cover; it extends to a 
ramified cover  $ s: \tilde N\to N$.

By construction  the maps
$\api(\tilde V_p)\to \api\bigl(\tilde N\setminus \tilde B\bigr) $
are all trivial  where $\tilde V_p $ denotes the preimage of $V_p$. 
Using (\ref{lgfg.say}.1) for $\tilde B\subset \tilde N$
 we obtain the following.
\medskip

{\it Claim \ref{lgfg.say}.3.} With the  above notation and assumptions
there is a natural isomorphism
$\api\bigl(\tilde N\setminus\tilde  B\bigr)\cong \api\bigl( \tilde N\bigr)$.
Furthermore, a similar isomorphism holds for every finite degree
 covering space of
$\tilde N\setminus\tilde  B $.
\qed
\medskip

(It is quite likely that the analog of (\ref{lgfg.say}.3)
fails for the topological fundamental group if $N$ is
an arbitrary simplicial complex. We do not know what happens
if $N$ is an algebraic variety.)
\end{say}

Applying this to a complex algebraic variety, we get the following
dichotomy which is essentially  proved in
\cite{gkp}. 

\begin{cor} \label{pi1.alternative.cior}
Let $X$ be a normal variety over $\c$ whose singularities are
potentially dlt. Then
\begin{enumerate}
\item either  $\api(X^{\rm sm})$ is finite
\item or there is a quasi-\'etale cover   $p:\tilde X\to X$
such that  $\api(\tilde X)=\api(\tilde X^{\rm sm})$ is infinite.
\end{enumerate}
\end{cor}

Proof.
By \cite{MR0377101}   there is a Whitney stratification  
$$
X_0:=X\supset X_2:=\sing X\supset X_3\supset \cdots  \supset X_{n+1}=\emptyset
$$
such that $X$ is  (locally in the Euclidean topology)
 a product along  each $X_i\setminus X_{i+1}$ and $\dim X_i=n-i$; 
see also \cite[Part I. Chap.1]{gm-book}.

We are done if $\api(X\setminus  X_2)$ is finite.
Otherwise, assume that 
$\api(X\setminus  X_i)$ is infinite for some $i$.
Let $(0\in Z_i)$ denote a general transversal slice of $X$ along 
$X_i\setminus X_{i+1}$. Then $(0\in Z_i)$ is potentially log terminal,
hence, by Theorem \ref{MR3187625.thm}, every point $x\in X_i\setminus X_{i+1}$ has a Euclidean open
 neighborhood $U_x$ such that 
$V_x:=U_x\setminus X_i$ is connected and
$\api(V_x)$ is finite. 

Choose a compactification $\bar X\supset X$
and apply  (\ref{lgfg.say}.3) with  $M:=\bar X$, $C:=X_{i+1}\cup(\bar X\setminus X)$  
and $B:=X_{i}$  to obtain
a  quasi-\'etale cover   $p:\tilde X\to X$ such that
$\api\bigl(\tilde X\setminus \tilde X_i\bigr)\cong \api\bigl( \tilde X\setminus \tilde X_{i+1}\bigr)$. 
Thus $\api\bigl( \tilde X\setminus \tilde X_{i+1}\bigr) $ is also infinite.

We can now replace $X$ by $\tilde X$ and apply the above argument
with $i$ replaced by $i+1$. After $n$ steps we get a  quasi-\'etale cover
$\tilde X^{(n)}\to X$ such that 
$\api\bigl( \tilde X^{(n)}\bigr) $ is infinite.\qed

\begin{say}[Proof of Proposition \ref{Xsm.sfg.fin.cor}] Let $p:\tilde X\to X$ be a  quasi-\'etale cover.
Then  $\bigl(\tilde X, \tilde \Delta:=p^*\Delta\bigr)$ is
also a dlt, logCY pair.
If there is  a dominant map $q:\tilde X\map Z$ onto a non-uniruled variety
then we are in the second case. Otherwise, as we noted in 
Paragraph \ref{rc.say.1},
 $\tilde X$ is rationally connected, hence simply connected.
Thus  $\api(X^{\rm sm})$ is finite by Corollary \ref{pi1.alternative.cior}. \qed
\end{say}

\section{Construction of Fano  models}

Our main technical theorem  says that 
every  logCY pair $(X, \Delta)$ with maximal intersection
is crepant birational
to  a logCY pair whose boundary is big.
For some applications one needs to know what happens 
without   maximal intersection and it is crucial to have 
very tight control over the exceptional divisors of the 
crepant birational equivalence.

\begin{thm}\label{p-bir} Let $(X, \Delta)$ be a
    $\q$-factorial, dlt,  logCY pair. Then there is a  crepant birational
map  $\phi: (X, \Delta)\map (\bar X, \bar \Delta)$ to a 
  logCY pair  and  a morphism
$q:\bar X\to Z$ such that 
\begin{enumerate}
\item $\bar\Delta^{=1}$ fully supports a $q$-ample divisor,
\item every log canonical  center of $(\bar X, \bar\Delta) $ dominates $Z$,
\item $\bar E\subset \bar\Delta^{=1}$
for every $\phi^{-1}$-exceptional divisor  $\bar E\subset \bar X$ and
\item $\phi^{-1}$ is an isomorphism over $\bar X\setminus \bar \Delta^{=1}$.
\end{enumerate}
\end{thm}

Note that in general $\bar X$ is neither $\q$-factorial nor dlt, but,
by (4), $\bar X\setminus \bar \Delta^{=1}$ is $\q$-factorial and klt.

\begin{say}[Outline of the proof]\label{p-bir-outline}
The proof,  inspired by \cite{BMSZ},  focuses on guaranteeing the properties (\ref{p-bir}.1--2), then we note
that (\ref{p-bir}.3--4) are also  achieved.

Assume that we have $(\bar X, \bar \Delta)$ and
let $\bar H$ be a $q$-ample divisor supported on $\bar\Delta^{=1}$.
Then $-\bigl(K_{\bar X}+\bar\Delta-\epsilon \bar H\bigr)$ is 
$q$-ample, thus $q$ is a Fano contraction for the pair
$\bigl(\bar X, \bar\Delta-\epsilon \bar H\bigr)$. 
It is thus reasonable to hope that we can obtain it using MMP for
$(X, \Delta-\epsilon H)$.  We do not know what $H$ should be,
so we start by pretending that $\bar H=\bar\Delta^{=1}$
and $\epsilon =1$.

{\it Step 1.} Run the $(X, \Delta^{<1})$-MMP. 
It ends with a Fano contraction
$g: (X_r, \Delta_r^{<1})\to Z_1$ where
$\Delta_r^{=1}\simq -\bigl(K_{X_r}+\Delta_r^{<1}\bigr)$ is $g$-ample.
Thus (\ref{p-bir}.1) holds but we have no information
on (\ref{p-bir}.2). 

{\it Step 2.} Apply the canonical bundle formula of
Paragraph \ref{kod.form.say} and write
$$
K_{X_r}+\Delta_r\simq 
g^*\bigl(K_{Z_1}+J_1+B_1).
$$
A key point is that  $(Z_1, J_1+B_1)$ behaves very much like a logCY pair.
(Conjecturally one can choose a divisor $D_1\simq J_1$ such that
$(Z_1, D_1+B_1)$ is a logCY pair.) 
Furthermore,  (\ref{kod.form.say}.10) essentially says that
 $B_1^{=1}\neq 0$ iff $ (X_r, \Delta_r)$ has a
 log canonical  center that does not  dominate $Z_1$.
(This actually holds only for suitable birational models.)

{\it Step 3.} Repeat Step 1 for $(Z_1, J_1+B_1)$
to get  $Z_1\map Z_2$.
(In general there are serious technical problems working with pairs $(Z, \Delta)$
where $\Delta$ is not known to be effective, but one can do this much.)

{\it Step 4.} Show that a suitable birational model of
the composite  $X\map Z_1\map  Z_2$ satisfies (\ref{p-bir}.1)
and repeat the arguments. At each step the dimension of $Z_i$ decreases, so eventually we stop and then (\ref{p-bir}.1--2) hold.

{\it Step 5.} Prove that the properties (\ref{p-bir}.3--4) also  hold.
\end{say}

 \begin{defn} Given  a  logCY pair $(Y, \Delta_Y)$ and  a dominant   morphism
$q:Y\to Z$, a log canonical center $W\subset Y$ is called
{\it vertical} if $q|_W:W\to Z$ is not dominant. 
\end{defn}


\begin{say}[Proof of Theorem \ref{p-bir}]\label{p-bir.pf.m}

Note that  $Z=\bar X=X$ works if  $( X, \Delta)$ is klt  and
Paragraph \ref{disconnect.say} settles the case when 
$\DMR( X, \Delta)$ is disconnected. Thus we may assume that
$\supp \Delta^{=1}$ is connected and nonempty.

Next we find a  logCY pair $(\tilde X, \tilde \Delta)$
that is crepant birational
to $( X, \Delta)$ and   a morphism
$q:\tilde X\to Z$ such that
\begin{enumerate}
\item  $\tilde\Delta^{=1}$ fully supports a divisor $\tilde H$
such that the corresponding map  $\psi:=\psi_{\tilde H}:\tilde X\map
\proj_Z \bigoplus_m q_*\bigl(\o_{\tilde X}(m\tilde H)\bigr)$
is  a morphism  outside the vertical log canonical centers
and an isomorphism  outside the  log canonical centers,
\item the properties  (\ref{p-bir}.3--4)  are satisfied and
\item $\dim Z$ is the smallest possible.
\end{enumerate}
(The trivial
case  $Z=\tilde X=X$ satisfies (\ref{p-bir.pf.m}.1),
(\ref{p-bir.pf.m}.3) and (\ref{p-bir.pf.m}.4)
 since any divisor is relatively ample for the
identity morphism. Thus such  $(\tilde X, \tilde \Delta)$ and $q:\tilde X\to Z$ exist.)

Note that  property (1) is clearly invariant under
 birational maps that are isomorphisms outside the vertical log canonical centers
(we will have many such maps during the proof).  
We then have to go from 
this property to  $q$-ampleness  at the end using 
 Lemma \ref{gen.amp.to.amp.lem}.

We are done if the property (\ref{p-bir}.2) also holds,
thus assume that there is a vertical log canonical center.
Next we first improve $\tilde X$, then $Z$ and finally obtain a contradiction by showing
that $\dim Z$ is not the smallest possible.

By assumption (\ref{p-bir}.4)  $\tilde X\setminus \tilde \Delta^{=1}$
is $\mathbb{Q}$-factorial and dlt, thus we can choose a
$\mathbb{Q}$-factorial dlt modification
$\pi: (\tilde X_1,\tilde \Delta_1)\to (\tilde X,\tilde \Delta)$
that is an isomorphism over $\tilde X\setminus \tilde \Delta^{=1}$.
Then $\pi^*H$ also satisfies property (1). 
Thus we may replace $(\tilde X,\tilde \Delta)$ by 
$(\tilde X_1,\tilde \Delta_1) $ and assume from now on that
$(\tilde X,\tilde \Delta)$ is $\mathbb{Q}$-factorial and dlt.


Next write $\tilde\Delta^{=1}=\Gamma^{\rm h}+\Gamma^{\rm v}$ as the sum of its  horizontal  and  vertical parts. 
By Lemma \ref{subcomplex.lemma},
after extracting some
divisors  we may also assume that $\Gamma^{\rm v}\neq 0$ and 
 $(\tilde X,\tilde \Delta^{<1}+\Gamma^{\rm h})$ has no vertical log canonical centers. (This is the only step that introduces  divisors that
appear in (\ref{p-bir}.3).) 

For the 3rd step,   we run a 
$\bigl(K_{\tilde X}+\tilde \Delta^{<1}+\Gamma^{\rm h}\bigr)$-MMP
using \cite[Thm.1.1]{MR3032329}.
After replacing $\tilde X$ with the resulting minimal model
and $Z$ by the corresponding  canonical model, 
may also assume that 
$K_{\tilde X}+\tilde \Delta^{<1}+\Gamma^{\rm h}$ is the pull-back of a 
$\q$-divisor from $Z$.  
Lemma \ref{where.isom.lem} guarantees that (\ref{p-bir}.4) still holds.
We can now apply the canonical bundle formula of
Paragraph \ref{kod.form.say} and write
$$
\begin{array}{lcl}
K_{\tilde X}+\tilde \Delta^{<1}+\Gamma^{\rm h} & \simq  &
q^*\bigl(K_Z+J+B^{\rm h})\qtq{and}\\
K_{\tilde X}+\tilde \Delta^{<1}+\Gamma^{\rm h}+\Gamma^{\rm v} & \simq  &
q^*\bigl(K_Z+J+B).
\end{array}
$$
Note that the $J$-parts are the same by 
(\ref{kod.form.say}.5) and the 
$\q$-linear equivalence $\Gamma^{\rm v}\simq q^*(B-B^{\rm h})$
implies  that
$\Gamma^{\rm v}=q^*(B-B^{\rm h})$. 
Furthermore, using (\ref{kod.form.say}.7--10) we see that
$B-B^{\rm h}$ equals the reduced part $B^{=1}$. Thus 
$\Gamma^{\rm v}=q^*(B^{=1})$.

So far we have focused on optimizing the choice of $\tilde X$; next
we consider $Z$. As we change $Z$, we have to keep changing $\tilde X$
to ensure that we have a morphism $\tilde X\to Z$.

As explained in 
Paragraph \ref{mmp.not.pseff.say}, we can run a
$(Z, J+B^{\rm h})$-MMP to get 
 $(\tilde Z,\tilde J+\tilde B^{\rm h}\bigr)$ and  $q_Z:\tilde Z\to W$.
Using Lemma \ref{lift.mmp.rel.logcy.lem} we see that, by passing to a suitable
crepant birational model of 
$\bigl(\tilde X, \tilde \Delta^{<1}+\Gamma^{\rm h}+\Gamma^{\rm v}\bigr)$
we may assume that we have a morphism  $\tilde q:\tilde X\to \tilde Z$
and $\Gamma^{\rm v}= q^*( B^{=1})$ remains true.
Furthermore, $Z\map \tilde Z$ is an isomorphism outside 
$\tilde Z\setminus \tilde B^{=1}$ by Lemma \ref{where.isom.lem}
and so (\ref{p-bir}.4) still holds by 
Lemma \ref{lift.mmp.rel.logcy.lem}.

We aim to use
$$
\bigl(\tilde X, \tilde \Delta=\tilde  \Delta^{<1}+\tilde \Gamma^{\rm h}+
\tilde \Gamma^{\rm v}\bigr)\qtq{and}
q_Z\circ  \tilde q:\tilde X\to \tilde Z\to W
$$ 
to contradict the minimality of $\dim Z$.  The question is
whether property (1) is satisfied or not. The problem is that there are
log canonical centers that are vertical for $q$ but not vertical for 
$q_Z\circ   \tilde q $.

We know that  $\tilde B^{=1}$ supports a  $q_Z$-ample divisor $H_Z$
and $\tilde \Delta^{=1}$ supports a  divisor $\tilde H$  whose restriction over $\tilde Z\setminus \tilde B^{=1}$ is the pull back of a relatively ample divisor  under a birational morphism.
Since $\tilde q^{-1}(\tilde B^{=1})=\supp \tilde \Gamma^{\rm v}\subset 
\supp \tilde \Delta^{=1}$, we see that
$\tilde H,   \tilde q^*H_Z$ are both supported on $\tilde \Delta^{=1}$.
If $\tilde H_X$ is  $\tilde q$-ample then $\tilde H_X+m \tilde q^*H_Z$ is  
$q_Z\circ \tilde q$-ample for $m\gg 1$ and we are done.

Finally, 
we aim to use Lemma \ref{gen.amp.to.amp.lem} 
for $A:=\tilde H$ 
to pass to another 
crepant birational model
$$
\bigl(\bar X, \bar  \Delta\bigr)\qtq{and}
\bar X\stackrel{\bar q}{\longrightarrow} \tilde Z\stackrel{ q_Z}{\longrightarrow} W
$$ 
where $\bar H_X$ is  $\bar q$-ample. 
Note that we can not apply 
Lemma \ref{gen.amp.to.amp.lem} directly to $\bigl(\tilde X, \tilde \Delta\bigr) $
since $\tilde \Gamma^{\rm v}$ contributes lc centers that do not
dominate $\tilde Z$. However, $\tilde \Gamma^{\rm v}$ is the pull-back of
$\tilde B^{=1}$, thus numerically $\tilde q$-trivial.
Therefore we can apply Lemma \ref{gen.amp.to.amp.lem}
to $\bigl(\tilde X, \tilde  \Delta^{<1}+\tilde \Gamma^{\rm h}\bigr)$
to get $ \bigl(\bar X, \bar  \Delta\bigr)\to \tilde Z$ as needed.

By property (1), $\psi_{H}$ is an isomorphism over
$\bar X\setminus \bar\Delta^{=1}$  thus
(\ref{p-bir}.4) still holds for the final $\bigl(\bar X, \bar  \Delta\bigr) $.
\qed
\end{say}

We have used several lemmas during the previous proof.

\begin{say}[MMP for $(Z, J+B)$] \label{mmp.not.pseff.say}
In general very little is known
about MMP for a pair $(Z, \Theta)$ where $\Theta$ is not effective
but  this can be done in the  special case when  $K_Z+\Theta$ is not
pseudo-effective and $(Z, \Theta)$ is a limit of klt pairs.

To make this precise, fix an ample divisor $H$ and assume that 
there is an effective $\q$-divisor $\Delta_{\epsilon}\simq
\Theta+\epsilon H $ such that 
$(Z, \Delta_{\epsilon})$ is klt and 
$K_Z+\Delta_{\epsilon}$ is still not
pseudo-effective. Then every step of  the $(K_Z+\Delta_{\epsilon})$-MMP with scaling of $H$ is also a step of  the $(K_Z+\Theta)$-MMP with scaling of $H$
and the program
ends with a Fano contraction   
$q:\bigl(Z_m, \Theta_m+\epsilon H_m\bigr)\to W$  such that
$-\bigl(K_{Z_m}+ \Theta_m\bigr)$ is $q$-ample. 

 We remark that  $Z$ does not to have to  be $\mathbb{Q}$-factorial. 
Given a birational contraction  $g_i:(Z_i,\Theta_i) \to W_i$, let
$Z_{i+1}$  be the canonical modification of $(W_i, (g_i)_*\Theta_i)$; the latter  exist by \cite{bchm}. The termination of this process   also follows from the finiteness of models as in \cite[Thm.E]{bchm}. 
\end{say}

The following  is essentially proved in \cite{birkar11, MR3032329}.

\begin{lem}\label{lift.mmp.rel.logcy.lem}
Let $(X,\Delta)$ be a $\mathbb{Q}$-factorial, quasi-projective, dlt pair with a 
projective morphism $f:X\to Z$ such that $f_*\o_X=\o_Z$. 
Assume that $K_X+\Delta\simq f^*N$ for a $\q$-Cartier $\q$-divisor $N$ on $Z$.
 Let $g:Z\to W$ be a birational morphism such that $g_*\o_Z=\o_W$ and assume 
that $R(Z/W, N):=\oplus_m g_*\o_Z(mN)$ is a finitely generated algebra over $W$.
 Set $Z^+:=\proj_W R(Z/W,N)$. Then there is  a model $X^+$ of $(X,\Delta)$ over 
$W$, such that 
\begin{enumerate}
\item $f$ extends to a  morphism $f^+: X^{ +}\to Z^{+}$ and
\item   $X\map X^+$ is  an isomorphism over the  locus where $Z\map W$ is 
an isomorphism. 
\end{enumerate}
\end{lem}

\begin{proof} It follows from \cite[2.12]{MR3032329} that $(X,\Delta)$ has a 
good minimal model $X^+$ over $W$. So $X^+$ admits a morphism to $Z^+$. 
Furthermore, by \cite[2.9]{MR3032329}, we know that $X^+$ can be obtained by 
running a suitable $(K_X+\Delta)$-MMP  over $W$. Thus all steps are
isomorphisms over the  locus where $Z\map W$ is 
an isomorphism. 
\end{proof}

\begin{lem} \label{where.isom.lem} 
Let $g:\bigl(X, \Delta_1+\Delta_2\bigr)\to Z$
be a  relative logCY where the $\Delta_i$ are effective and $\Delta_2$
is $\q$-Cartier.  Let $\phi:X\map X'$ be a rational map
obtained by running a $(K_X+\Delta_1)$-MMP. Then 
$\phi^{-1}$ is an isomorphism on $X'\setminus \supp\Delta_2'$. 
\end{lem}

Proof. It is enough to check this when $\phi$ is  a single step of the MMP.
So assume that $\phi$ corresponds to the extremal ray 
$R\subset \overline{NE}(X/Z)$. 
Let $g:X\to W$ denote the corresponding contraction
and set
$$
X':= \proj_W  \sum_{m\geq 0}\o_W\bigl(\rdown{mK_W+mg_*\Delta_1}\bigr)=
 \proj_W  \sum_{m\geq 0}\o_W\bigl(\rdown{-mg_*\Delta_2}\bigr).
$$
 If a curve  $C'\subset X'$  is contracted by $g'$ then
$\bigl(C'\cdot \Delta'_2\bigr)<0 $, thus
$\ex(g')\subset \supp \Delta'_2$. Therefore
$g'$ is an isomorphism on $X'\setminus  \supp\Delta_2'$.

Pick $w\in W$. If $g'$ is a local isomorphism over $w$ but $g$ is not
then there is a $g$-exceptional divisor $E$ such that $w\in g(E)$. 
Let $C\subset X$ be an irreducible
curve such that $[C]\in R$. Then 
$\bigl(C\cdot (K_X+\Delta_1)\bigr)<0$ hence
$\bigl(C\cdot \Delta_2\bigr)>0$. In particular $E\subset \supp  \Delta_2$.
An effective exceptional divisor can not be relatively nef, thus
$w\in \supp(g_*\Delta_2)$. Therefore 
$\phi^{-1}$ is an isomorphism on $X'\setminus \supp\Delta_2'$. 
\qed


\medskip

\begin{lem} \label{gen.amp.to.amp.lem}  
Let $g:(X, \Delta)\to Z$
be a relative logCY and  $A$  a $\q$-Cartier $\z$-divisor
on $X$. Let $Z^0\subset Z$ be an open subset such that 
\begin{enumerate}
\item the restriction of $|A|$ to $X_0:=g^{-1}(Z_0)$ is base point free and big,
\item   no lc center of $(X, \Delta) $
is contained in $g^{-1}(Z\setminus Z^0)$. 
\end{enumerate}
Then there is a relative logCY
$g':(X', \Delta')\to Z$ and a crepant birational contraction map
$\phi:(X, \Delta)\map (X', \Delta') $  such that $\phi_*A$ is  $q'$-ample.
\end{lem}

Proof.  Let $\pi\colon \tilde X\to X$ be a $\mathbb{Q}$-factorial dlt modification of $(X,\Delta)$ with $K_{\tilde{X}}+\tilde \Delta=\pi^*(K_X+\Delta)$.
Let  $D\in \pi^*|A|$ 
  a general member.  Then $D$ does not contain any of the log canonical centers
of $( \tilde X,  \tilde \Delta)$, hence $( \tilde X, \tilde \Delta+\epsilon D)$
is also dlt for $0<\epsilon\ll 1$. 
By \cite{MR3032329} it has a relative  canonical model 
$( X',  \Delta'+\epsilon D')$. Then $D'$ is $g'$-ample since $K_{X'}+\Delta'$ is numerically $g'$-trivial and  $D'\sim_{g} \phi_*A$. \qed
\medskip

\begin{lem}\label{subcomplex.lemma}
Let $(Y,\Delta)$ be a $\q$-factorial dlt pair and   
${\mathcal C}\subset \DMR(Y,\Delta)$ a closed subcomplex 
that contains all the vertices.
Let
$W_C\subset Y$ be the union of the strata
corresponding to those simplices that are not contained in 
${\mathcal C}$.  

Then there is  a $\q$-factorial dlt pair  $(Y_C,\Delta_C)$ and
a proper, crepant, birational morphism  $g:(Y_C,\Delta_C)\to (Y,\Delta)$
such that 
\begin{enumerate}
\item   $g$ is an isomorphism over $Y\setminus W_C$,
\item all the  exceptional divisors have discrepancy $-1$ and
\item $\DMR\bigl(Y_C, g^{-1}_*\Delta\bigr)={\mathcal C}$.
\end{enumerate}
\end{lem}

 \begin{proof}
Assume first that $(Y,\Delta)$  is a simple normal crossing pair.
Then the process is straightforward. First we blow up the 
strata corresponding the  maximal dimensional simplices   not contained in 
${\mathcal C}$, then the strata corresponding to codimension 1  simplices 
  not contained in 
${\mathcal C}$ and so on.

In general, first we take a thrifty resolution 
$(Y_1,\Delta_1)\to(Y,\Delta)$ 
as in \cite[Sec.2.5]{kk-singbook} where  $\Delta_1 $ is the birational 
transform of  $\Delta $.
 Then   $(Y_1,\Delta_1)$ is a
simple normal crossing pair, thus we can apply the previous argument
to get $(Y_2,\Delta_2)\to(Y_1,\Delta_1)\to(Y,\Delta)$
  where  $\Delta_2 $ is the birational 
transform of  $\Delta_1 $.
Finally we take a minimal model for 
  $(Y_2,\Delta_2+E)$, where $E$ is the exceptional divisor of 
$Y_2\to Y$. As in \cite[1.35]{kk-singbook}, this removes the
exceptional divisors whose discrepancy is $>-1$.
\end{proof}

In general  the pair $(\bar X, \bar \Delta)$ in Theorem \ref{p-bir}
is neither unique 
nor dlt and for some applications we need a carefully chosen dlt modification.

Start with a projective logCY pair $(X, \Delta)$.
 After  extracting  divisors with discrepancy $-1$ 
 we may assume  that $(X, \Delta) $ is dlt and $\q$-factorial.
 Then apply Theorem \ref{p-bir} to get $(\bar X, \bar \Delta)$ and
$q:\bar X\to Z$.
Finally let $g:(X_s, \Delta_s)\to (\bar X, \bar \Delta)$ be a 
$\q$-factorial, dlt, crepant modification that extracts
precisely all $\phi$-exceptional divisors and possibly some other
divisors with discrepancy $-1$; see \cite[1.38]{kk-singbook}.
 Then
$g^{-1}\circ \phi$ is a crepant birational map
and we obtain the following.

\begin{cor}\label{p-bir-cor}
 Every  logCY pair $(X, \Delta)$ is crepant birational
to  a $\q$-factorial, dlt, logCY pair $(X_s, \Delta_s)$
with crepant birational maps   
$$
\psi:=g^{-1}\circ \phi: (X, \Delta)\stackrel{\phi}{\map} (\bar X, \bar \Delta)
\stackrel{g}{\leftarrow} (X_s, \Delta_s) 
$$ 
such that there is a morphism
$q_s:X_s\to Z$ with the following properties.
\begin{enumerate}
\item Every log canonical  center of $(X_s, \Delta_s) $ dominates $Z$,
\item $\Delta_s^{=1}\subset g^{-1}(\bar \Delta^{=1})$,
\item $g^{-1}(\bar \Delta^{=1})$ fully supports a $q_s$-big and $q_s$-semi-ample divisor,
\item every prime divisor $E_s\subset X_s$
has non-empty intersection with $g^{-1}(\bar \Delta^{=1})$.
\item $\psi^{-1}$ is a crepant, birational contraction and $ E_s\subset \Delta_s^{=1}$
for every $\psi^{-1}$-exceptional divisor  $ E_s\subset  X_s$. \qed
\end{enumerate}
\end{cor}

\section{The boundary of crepant birational pairs}

So far we have studied the dual complex of logCY pairs  $(X, \Delta)$. 
Some of the invariance results have their counterparts for
$\rdown{\Delta}$ as well.

\begin{thm} \label{boundary.coh.pi1.thm}     
Let $g:(X_1, \Delta_1)\map (X_2, \Delta_2)$ 
be a proper, crepant birational map
of dlt log pairs as in (\ref{crep.eq.def.1}). 
Assume that either the $\Delta_i$ are effective or
the $(X_i, \Delta_i) $ are simple normal crossing pairs.
Set $D_i:=\Delta^{=1}_i$.
Then
\begin{enumerate}
\item $H^i\bigl(D_1, \o_{D_1}\bigr)\cong 
H^i\bigl(D_2, \o_{D_2}\bigr)$ for every $i$ and 
\item 
$\pi_1\bigl(D_1\bigr)\cong \pi_1\bigl(D_2\bigr)$.
\end{enumerate}
\end{thm}

Proof.  
In the  simple normal crossing case
 $g$ can be factored as a composite of smooth blow-ups and 
blow-downs whose centers have simple normal crossings with  $  \supp\Delta$.
For any one such blow-up the claimed isomorphisms 
(\ref{boundary.coh.pi1.thm}.1--2) are clear. 

In order to reduce to this case 
let $ (X, \Delta)$ be a dlt log pair with $\Delta$ effective; thus
$\rdown{\Delta}=\Delta^{=1}$. Let
$p:(Y, \Delta_Y)\to (X, \Delta)$ a thrifty resolution as in 
 \cite[Sec.2.5]{kk-singbook}.  Set $D:=p^{-1}_*\Delta^{=1}$.
Consider the exact sequence
$$
0\to \o_{Y}(-D)\to \o_{Y}\to
 \o_{D}\to 0.
$$
Note that $R^ip_* \o_{Y}=0$ for $i>0$ since $X$ has rational singularities
and  \cite[2.87]{kk-singbook} implies that 
 $R^ip_* \o_{Y}(-D)=0$ for $i>0$. 
Thus $R^ip_*  \o_{D}=0$ for $i>0$ and
$p_*  \o_{D}= \o_{\rdown{\Delta}}$. Therefore
 $H^i\bigl(D, \o_{D}\bigr)\cong 
H^i\bigl(\rdown{\Delta}, \o_{\rdown{\Delta}}\bigr)$.

It remains to establish that (\ref{boundary.coh.pi1.thm}.2) holds for $p$.
For this  it is enough to prove that
every fiber of  $g:=p|_D:D\to \rdown{\Delta}$ 
is simply connected.
If $\rdown{\Delta}$ is normal, this follows from
\cite[7.5]{k-shaf} in the algebraic case and \cite{tak} in general.
The proof of \cite[7.5]{k-shaf} works in the non-normal case;
most likely the same applies to \cite{tak} but we have not checked the details.
One can, however, go around this problem as follows. 

Fix a point $x\in \rdown{\Delta}$ and let $S_1,\dots, S_r\subset X$ be the 
irreducible components of
$\rdown{\Delta}$ passing through $x$ with birational transforms
 $D_1,\dots, D_r\subset Y$. 

We use induction on $r$ and on the 
dimension. If $r=1$ then, as we noted,  $g^{-1}(x)$ is  simply connected.
For $r>1$, let $h$ denote the restriction of $g$ to $D_r$
and $g_{r-1}$ the restriction of $g$ to  $D_1\cup\cdots\cup D_{r-1}$.
Both  $h^{-1}(x)$ and  $g_{r-1}^{-1}(x)$ are  simply connected by induction.
Their intersection is isomorphic to the fiber we get from the
lower dimensional  thrifty resolution
$$
\bigl(D_r, \diffg_{D_r}\Delta_Y\bigr)\to \bigl(S_r, \diffg_{S_r}\Delta\bigr),
$$
hence also simply connected. Thus $g^{-1}(x)$ is also  simply connected
by van~Kampen's theorem.
 \qed

\section{Examples}\label{sec.examples}

\begin{exmp} Start with  $X_1=\p^1$ and $\Delta=(0{:}1)+(1{:}0)$.
Let $\tau_1$ be the involution  $(x{:}y)\mapsto (y{:}x)$.
Set $(X_n, \Delta_n):= (X_1, \Delta_1)^n$,
$\tau_n$ the involution $(\tau_1,\dots, \tau_1)$
and $(Y_n, \Delta^Y_n):=(X_n, \Delta_n)/(\tau_n)$.

Then $\DMR(X_n, \Delta_n)$ is the boundary of the cube of dimension $n$, thus
PL-homeomorphic to ${\mathbb S}^{n-1}$ while
$\DMR(Y_n, \Delta^Y_n)\simeq \r\p^{n-1}$.
\end{exmp}

\begin{exmp} Let $ G$ be  a group of order $m$. 
It acts on $\p^{m-1}$ by permuting the coordinates.
Pick a general hyperplane and move it around by $G$ to get a
logCY pair  $(\p^{m-1}, \Delta_m)$.  The dual complex is the
boundary of the $(m-1)$-simplex,  thus
PL-homeomorphic to ${\mathbb S}^{m-2}$.

Next we take the quotient by $G$. The boundary consists of  1 divisor with
complicated self-intersections. It is thus better to take the
 barycentric subdivision first and then take the quotient. 
See \cite[Rem.10]{dkx} on how to do this with blow-ups to obtain
a dlt pair  $(X_m, \Delta_m)$.
The resulting quotient need not be dlt; such examples led to the
introduction of quotient-dlt pairs in  \cite[Sec.5]{dkx}.

The dual complex of the quotient is ${\mathbb S}^{m-2}/G$.
Note that usually $G$ has fixed points on ${\mathbb S}^{m-2}$.
The only exception occurs when $G$ is cyclic of prime order $m=p$.
In these cases
$\pi_1\bigl(\DMR(X_p, \Delta_p))\cong \z_p$.
\end{exmp}

\begin{exmp}\label{pi1.not.birinv.exmp}
Let $A$ be an Abelian surface and $X=A/(\pm)$ the corresponding
Kummer surface with minimal resolution $ X'$. Then
$X$ and $X'$ are crepant-birational. Note that $X'$  is smooth
and simply connected while  $\pi_1\bigl(X^{\rm sm}\bigr)$ is infinite.

There are similar examples involving only rational surfaces.
Start with $\p^1\times\p^1$ and an involution $\tau$  with 4 isolated fixed points. Set $X=(\p^1\times\p^1)/(\tau)$  with minimal resolution $ X'$.
In this example $X'$  is smooth
and simply connected while  $\pi_1\bigl(X^{\rm sm}\bigr)\cong \z/2$.
\end{exmp}

\begin{exmp}\label{big.nef.not.rc.exmp} Let $Y$ be a smooth
CY and $L$ a very ample line bundle. Set $X:=\p_Y(L+\o_Y)$
and $D, D'\subset X$ the 2 sections. 
Then
$(X, D+D')$ is a logCY.

Choose indexing such that $\o_X(D)|_{D}\cong L$.
Then $D$ is big and semi-ample. Thus $D+D'$ supports---but not fully supports---a big and semi-ample divisor and yet $X$ is not rationally connected.

The linear system  $|D|$ is base point free; let
$D_1,\dots, D_n$ be general members. Set
$\Delta:=\frac1{n}(D_1+\cdots + D_n)+D'$. Then
$(X,\Delta)$ is a log CY with a morphism
$X\to Y$ to  a CY  but, for $n\geq 3$,  $(X,\Delta)\to Y$
is not a product, not even locally analytically at the generic point.

This contrasts with the product theorem of \cite{k-lars}
for Calabi--Yau varieties without boundary divisors.
\end{exmp}

\begin{exmp} \label{join.and.product.exmp}
Let $(X_i, \Delta_i)$ be two logCY pairs.
The product
$$
(X_1, \Delta_1)\times (X_2, \Delta_2):=
\bigl(X_1\times X_2, X_1\times \Delta_2+\Delta_1\times X_2\bigr)
$$
is also logCY and
$$
\DMR\bigl((X_1, \Delta_1)\times (X_2, \Delta_2)\bigr)\simeq
\DMR(X_1, \Delta_1)* \DMR(X_2, \Delta_2)
$$
where $*$ denotes the join.
Thus if $\DMR(X_i, \Delta_i)\simeq {\mathbb S}^{n_i}/G_i$
then
$$
\DMR\bigl((X_1, \Delta_1)\times (X_2, \Delta_2)\bigr)\simeq
{\mathbb S}^{n_1+n_2+1}/G_1\times G_2.
$$
\end{exmp}

\def\cprime{$'$} \def\cprime{$'$} \def\cprime{$'$} \def\cprime{$'$}
  \def\cprime{$'$} \def\cprime{$'$} \def\dbar{\leavevmode\hbox to
  0pt{\hskip.2ex \accent"16\hss}d} \def\cprime{$'$} \def\cprime{$'$}
  \def\polhk#1{\setbox0=\hbox{#1}{\ooalign{\hidewidth
  \lower1.5ex\hbox{`}\hidewidth\crcr\unhbox0}}} \def\cprime{$'$}
  \def\cprime{$'$} \def\cprime{$'$} \def\cprime{$'$}
  \def\polhk#1{\setbox0=\hbox{#1}{\ooalign{\hidewidth
  \lower1.5ex\hbox{`}\hidewidth\crcr\unhbox0}}} \def\cdprime{$''$}
  \def\cprime{$'$} \def\cprime{$'$} \def\cprime{$'$} \def\cprime{$'$}
\providecommand{\bysame}{\leavevmode\hbox to3em{\hrulefill}\thinspace}
\providecommand{\MR}{\relax\ifhmode\unskip\space\fi MR }
\providecommand{\MRhref}[2]{%
  \href{http://www.ams.org/mathscinet-getitem?mr=#1}{#2}
}
\providecommand{\href}[2]{#2}

\medskip

\noindent JK: Princeton University, Princeton NJ 08544-1000

{\begin{verbatim} kollar@math.princeton.edu\end{verbatim}}
\medskip

\noindent  CX: Beijing International Center of Mathematics Research, 
Beijing, 100871, China  

{\begin{verbatim} cyxu@math.pku.edu.cn\end{verbatim}}

\end{document}